\documentclass[12pt,reqno]{amsart}
\usepackage{graphicx}
\usepackage{verbatim}
\usepackage{textcomp}
\usepackage{amssymb}
\usepackage{cite}
\usepackage{amsmath}
\usepackage{latexsym}
\usepackage{amscd}
\usepackage{amsthm}
\usepackage{mathrsfs}
\usepackage{xypic}
\usepackage{bm}
\usepackage{url}
\usepackage{hyperref}
\usepackage{amsfonts}
\usepackage{times,graphicx,hyperref,mathrsfs}
\usepackage{CJK}
\usepackage{color}
\newtheorem{thm}{Theorem}[section]

\newtheorem{lem}[thm]{Lemma}
\newtheorem{prop}[thm]{Proposition}

\theoremstyle{definition}
\newtheorem{defn}{Definition}[section]

\theoremstyle{remark}
\newtheorem{rem}{\bf Remark}[section]
\numberwithin{equation}{section}
\def\supp{{\rm{\,supp\,}}}

\allowdisplaybreaks
\begin{document}
\begin{sloppypar}
\title[{Pseudo-differential operators}]
{\uppercase{Sharp maximal function estimates and $H^{p}$ continuities of pseudo-differential operators}}

\author{Guangqing Wang}

\address{ School of Mathematics and Statistics, Fuyang Normal University, Fuyang, Anhui 236041, P.R.China}
\thanks{The research of author is supported in part by Scientific Research Foundation of Education Department of Anhui Province of China (2022AH051320,KJ2021A0659), Doctoral Scientific Research Initiation Project of Fuyang Normal University (2021KYQD0001) and University Excellent Young Talents Research Project of Anhui Province (gxyq2022039).}
\email{wanggqmath@whu.edu.cn}

\maketitle

\begin{abstract} It is studied that pointwise estimates and continuities on Hardy spaces of pseudo-differential operators (PDOs for short) with the symbol in general H\"{o}rmander's classes. We get weighted weak-type $(1,1)$ estimate, weighted normal inequalities, $(H^{p},H^{p})$ continuities and $(H^{p},L^{p})$ continuities for PDOs, where $0<p\leq1$.
\end{abstract}



\section{Introduction and main results}
Let $m\in \mathbb{R}$, $0\leq\varrho,\delta\leq1$.  A symbol $a(x,\xi)$ is said to be in the H\"{o}rmander class $S^{m}_{\varrho,\delta}$ \cite{Hormander2}, if $a(x,\xi)\in C^{\infty}(\mathbb{R}^{n}\times\mathbb{R}^{n})$ with
\begin{eqnarray*}
|\partial^{\beta}_x\partial^{\alpha}_{\xi}a(x,\xi)|\leq C_{\alpha,\beta}\langle\xi\rangle^{m-\varrho |\alpha|+\delta|\beta|},
\end{eqnarray*}
for any multi-indices $\alpha,\beta.$
The pseudo-differential operators with symbol $a(x,\xi)$ is defined by the formula
\begin{eqnarray}\label{df}
T_{a}u(x)
&=&\frac{1}{(2\pi)^n}\int_{\mathbb{R}^n} e^{ i \langle x,\xi\rangle}a(x,\xi)  \hat{u}(\xi)d\xi,
\end{eqnarray}
where $\hat{u}$ denotes the fourier transform of $u$. An important topic on the pseudo-differential operators is to study the properties of these operators acting on some function spaces and some pointwise estimates for them. $L^{p}$ regularity is a fundamental one which can be gotten by the complex interpolation between $L^{2}$-continuity and ($L^{\infty},BMO$)-continuity, see \cite{Fefferman,Stein,Ragusa3}. As we know, $L^{2}$-continuity of the pseudo-differential operators is sharp in terms of its order $m\leq-\frac{n}{2}\max\{\delta-\varrho,0\}$, where $0\leq\varrho\leq1$ and $0\leq\delta<1$, see \cite{Hounie2,Hormander3}. However, it is not clear if the ($L^{\infty},BMO$)-continuity is sharp when $0\leq\varrho<\delta<1$, see \cite{Kenig,Miyachi}. On the one hand, if $a(x,\xi)\in L^{\infty}S^{m}_{\varrho}$ with $m<-\frac{n}{2}(1-\varrho)$, the pseudo-differential operators are bounded on $L^{\infty}(\mathbb{R}^n)$ \cite{Kenig}, which implies the ($L^{\infty},BMO$)-continuity. Here $L^{\infty}S^{m}_{\varrho}$ denotes the rough H\"{o}rmander class whose constituent $a(x,\xi)$ obeys
 \begin{equation*}
   \|\partial^{\alpha}_{\xi}a(\cdot,\xi)\|_{L^{\infty}(\mathbb{R}^n)}\leq C_{\alpha}\langle\xi\rangle^{m-\varrho |\alpha|}.
\end{equation*}
Clearly, the relation $S^{m}_{\varrho,\delta}\subset L^{\infty}S^{m}_{\varrho}$ holds for any $m\in \mathbb{R}$, $1\leq\varrho,\delta\leq1$. On the other hand, there is a symbol a $a\in S^{m}_{\varrho,0}$ such that $T_{a}$ dose not map $L^{\infty}$ to ${\rm BMO}$ if $m>-\frac{n}{2}(1-\varrho)$, see \cite{Miyachi}.

Recently, taking full advantage of the smooth of variate $x$, the author \cite{W} prove that if $0\leq\varrho\leq1$, $0\leq\delta<1$ and $a(x,\xi)\in S^{-\frac{n}{2}(1-\varrho)}_{\varrho,\delta}$, the $(L^{\infty},BMO)$-continuity of the pseudo-differential operators $T_{a}$ is true, and clearly it is sharp. Moreover, the $L^{p}$ boundedness is studied as well.
\begin{thm}[Wang \cite{W}]\label{Tz}
Let $1<p<\infty$, $0\leq\varrho\leq1,$ $0\leq\delta<1$ and $a(x,\xi)\in S^{m}_{\varrho,\delta}$.
If $$m\leq -n(1-\varrho)|\frac{1}{2}-\frac{1}{p}|-n\frac{\max\{\delta-\varrho,0\}}{\max\{p,2\}},$$
then
$$\|T_{a}u\|_{L^{p}}\lesssim\|u\|_{L^{p}}.$$
\end{thm}
Clearly, the range of $m$ in \cite[Theorem 3.4]{Hounie} is revised when $2<p<\infty$ and $0\leq\varrho<\delta<1$. For the case $1<p<\infty$ and $0\leq\delta\leq\varrho<1$, we refer to \cite{Hormander3,Wang,Stein}.

It is a pity that the main idea is inapplicable to its dual operators $T^{*}_{a}$ which is defined by the formula
\begin{eqnarray}\label{df*}
T^{*}_{a}u(x)
&=&\frac{1}{(2\pi)^n}\int_{\mathbb{R}^n}\int_{\mathbb{R}^n} e^{ i \langle x-y,\xi\rangle}a(y,\xi)d\xi u(y) dy.
\end{eqnarray}
\allowdisplaybreaks
So the $(L^{\infty},BMO)$-continuity of $T^{*}_{a}$ has been understood so far\cite{Hounie} only if $a(y,\xi)\in S^{-\frac{n}{2}(1-\varrho)-\frac{n}{2}\max\{\delta-\varrho,0\}}_{\varrho,\delta}$. However, one can get $(H^{1},L^{1})$-continuity of $T^{*}_{a}$ under the condition $a(y,\xi)\in S^{-\frac{n}{2}(1-\varrho)}_{\varrho,\delta}$ (see Theorem \ref{TH1}). By complex interpolation, we have
\begin{thm}\label{Tz}
Let $1<p<\infty$, $0\leq\varrho\leq1,$ $0\leq\delta<1$ and $a(x,\xi)\in S^{m}_{\varrho,\delta}$.
If $$m\leq -n(1-\varrho)|\frac{1}{2}-\frac{1}{p}|-n\max\{\delta-\varrho,0\}(1-\frac{1}{\min\{p,2\}}),$$
then
$$\|T^{*}_{a}u\|_{L^{p}}\lesssim\|u\|_{L^{p}}.$$
\end{thm}
In this paper, the properties of pseudo-differential operator acting on Hardy spaces $H^{p}(\mathbb{R}^{n})$ that is a right replacement for $L^{p}(\mathbb{R}^{n})$ when $0<p\leq1$, and some pointwise estimates for these operators are investigated.
Clearly, the $L^{p}(p\neq2)$ continuity between $T_{a}$ and $T^{*}_{a}$ is different in terms of the order $m$. Based on this observation, both $T_{a}$ and $T^{*}_{a}$ will be considered in this paper.

For the sake of narration, it is necessary to introduce some notations firstly. For a function $u\in L^{1}_{loc}(\mathbb{R}^{n})$, we define the Fefferman-Stein sharp maximal
function and Hardy-Littlewood maximal function by the formula:
$$M^{\sharp}u(x)=\sup\limits_{x\in Q}\inf\limits_{c}\frac{1}{|Q|}\int_{Q}|u(y)-c|dy\quad {\rm and}\quad Mu(x)=\sup\limits_{x\in Q}\frac{1}{|Q|}\int_{Q}|u(y)|dy$$
respectively, where $c$ moves over all complex number, and $Q$ containing $x$ moves over all cubes with its sides parallel to the coordinate axes. For $\epsilon>0$, denote $M^{\sharp}_{\epsilon}u=\big(M^{\sharp}(|u^{\epsilon}|)\big)^{1/\epsilon}$ and $M_{\epsilon}u=\big(M(|u^{\epsilon}|)\big)^{1/\epsilon}$.

The pointwise estimate of pseudo-differential operators in terms of $M^{\sharp}$ and $M$ are given by many authors. For example, Chanillo and Torchinsky \cite{Chanillo}, Journ\'{e} \cite{Journe}, Miller \cite{Miller}, Miyachi \cite{Miyachi}, Wang and Chen \cite{CW}, Park and Tomita \cite{Park} and Wang \cite{W} and so on. We refer to \cite{Beltran,Bagchi1,Bagchi2} for the pointwise sparse bounds of these operators. Here, one is apt to state a result by Miyachi and Yabuta \cite{MiyachiY}.
\begin{thm}[Miyachi and Yabuta \cite{MiyachiY}]\label{MY}
Let $1<p\leq2$, $0<\varrho\leq\frac{p}{2}$ and $\varrho<1$. If $a\in S^{-n(1-\varrho)/p}_{\varrho,\varrho}$, then
$$M^{\sharp}(T_{a}f)(x)\lesssim M_{p}f(x).$$
\end{thm}
Clearly, there is a restriction on the range of $\varrho,\delta$ and $p$, that is $0<\varrho=\delta\leq\frac{p}{2}$ with $\varrho<1$ and $p\neq1$. Recently, this restriction on $\varrho,\delta$ is extended to $0\leq\varrho=\delta<1$ by Park and Tomita \cite{Park,Park3} and to $0\leq\varrho\leq1,$ $0\leq\delta<1$ when $p=2$ in \cite{W}. However, the case of $1<p<2$, $0\leq\varrho\leq1,$ $0\leq\varrho<\delta<1$ and $p=1$, $0\leq\varrho\leq1,$ $0\leq\delta<1$ seems to be not clear. Particularly, there is no corresponding result in case $p=1$, but a weaker version is obtained by Michalowski, Rule and Staubach \cite{Michalowski}.
\begin{thm}[Michalowski, Rule and Staubach \cite{Michalowski}]\label{MRS}
Let $0<\varrho\leq1$, $0\leq\delta<1$ and $1<p<\infty$. If $a\in S^{-n(1-\varrho)}_{\varrho,\delta}$, then
$$M^{\sharp}(T_{a}f)(x)\lesssim M_{p}f(x).$$
\end{thm}

The first main result of this paper is a generalization of Theorem \ref{MY}. And the operator $T^{*}_{a}$ is considered as well.

\begin{thm}\label{TH2}
Let $0\leq\varrho\leq1$, $0\leq\delta<1$ and $1<p\leq2$. If $a\in S^{-n(1-\varrho)/p}_{\varrho,\delta}$, then
$$M^{\sharp}(T_{a}f)(x)\lesssim M_{p}f(x).$$
If $a\in S^{-\frac{n}{p}(1-\varrho)-\frac{n}{2}\max\{\delta-\varrho,0\}}_{\varrho,\delta}$, then
$$M^{\sharp}(T^{*}_{a}f)(x)\lesssim M_{p}f(x).$$
\end{thm}

The second main result of this paper is extending $p$ in Theorem \ref{MRS} to the extreme case $p=1$.
\begin{thm}\label{TH3}
Let $0<\varrho\leq1$, $0\leq\delta<1$ and $0<\epsilon<1$. If $a\in S^{-n(1-\varrho)}_{\varrho,\delta}$, then
\begin{equation*}\label{MM}
M^{\sharp}_{\epsilon}(T_{a}f)(x)\lesssim Mf(x)\quad  and  \quad M^{\sharp}_{\epsilon}(T^{*}_{a}f)(x)\lesssim Mf(x).
\end{equation*}
\end{thm}
Interesting that the order of $T^{*}_{a}$ in Theorem \ref{TH2} seem to be improved when $p=1$. It is not cleat that if the order of $T^{*}_{a}$ can be improved in the case $1<p\leq2$. Another interesting thing is that the second estimate holds with $a\in L^{\infty}S^{-n(1-\varrho)}_{\varrho}$ in case $0<\varrho<1$.
\begin{thm}\label{TH5}
Let $0<\varrho<1$.
If $a\in L^{\infty}S^{-n(1-\varrho)}_{\varrho}$ then
$$M^{\sharp}(T^{*}_{a}f)(x)\lesssim Mf(x).$$
\end{thm}

As we know, the pointwise estimates can give some weighted $L^{p}$ inequalities.  Recall that a
nonnegative locally integrable function $\omega$ belongs to the class of Muckenhoupt $A_{p}$ weights if there exists a constant $C>0$ such that

\begin{eqnarray}
&\sup\limits_{Q\subset\mathbb{R}^{n}}\big(\frac{1}{|Q|}\int_{Q}\omega(x)dx\big)
\big(\frac{1}{|Q|}\int_{Q}\omega(x)^{\frac{1}{1-p}}dx\big)^{p-1}\leq C\quad \mathrm{when} \quad 1<p<\infty;\label{Ap}\\
&M\omega(x)\leq C\omega(x) \quad \mathrm{~for ~almost~ all~} x\in\mathbb{R}^{n} \quad \mathrm{when} \quad p=1.\label{A1}
\end{eqnarray}
For $p=\infty$, one define $A_{\infty}:=\cup_{p>1}A_{p}$. The smallest constant appearing in (\ref{Ap}) or (\ref{A1}) is called the $A_{p}$ constant of $\omega$ which is denoted by $[\omega]_{p}$. The usual notation that
$$\|u\|^{p}_{L^{p}_{\omega}}=\int_{\mathbb{R}^{n}}|u(x)|^{p}\omega(x)dx~\mathrm{and}~
\|u\|^{p}_{L^{p,\infty}_{\omega}}=\sup_{\lambda>0}\lambda^{p}\omega(x\in\mathbb{R}^{n}:|u(x)|>\lambda) $$
will be adopted in this paper. The weighted $L^{p}$ estimates for pseudo-differential operators has been a topic extensively studied, specially in the 1980s \cite{Chanillo,MiyachiY,Journe,Hounie}, later improved by Michalowski et.al \cite{Michalowski,Michalowski1} in the late 2000s and revisited in \cite{W,Park} recently.
\begin{thm}[Wang \cite{W}]\label{W}
Let $0\leq\varrho\leq1,$ $0\leq\delta<1$, $1\leq r\leq2$ and $a(x,\xi)\in S^{-\frac{n}{r}(1-\varrho)}_{\varrho,\delta}$. Suppose $\omega\in A_{p/r}$ with $r<p<\infty$. Then there is a constant $C$ independent of $a$ and $u$, such that
\begin{equation}\label{weighted}
\|T_{a}u\|_{L^{p}_{\omega}}\leq C\|u\|_{L^{p}_{\omega}}.
\end{equation}
\end{thm}
Theorem \ref{W} is proved by some interpolations between $r=1$ and $r=2$. In this paper, a new proof will be given.

By  interpolation theory \cite[Theorem 5.5.3]{Bergh} and the famous Fefferman-Stein's inequalities \cite{Stein1}, that is,
$$\|M_{\epsilon}u\|_{L^{p}_{\omega}}\lesssim \|M^{\sharp}_{\epsilon}u\|_{L^{p}_{\omega}},\quad \|M_{\epsilon}u\|_{L^{p,\infty}_{\omega}}\lesssim \|M^{\sharp}_{\epsilon}u\|_{L^{p,\infty}_{\omega}}$$
for $0<\epsilon,p<\infty$ and $\omega\in A_{\infty}$,
Theorem \ref{TH2}, Theorem \ref{TH3} and Theorem \ref{TH5} lead to the following weighted $L^{p}$ inequalities.
\begin{thm}\label{weight2}
Let $0\leq\varrho\leq1,$ $0\leq\delta<1$ and $1\leq r\leq2$. For any $r\leq p<\infty$ $(1<p<\infty$ if $r=1)$ and $\omega\in A_{p/r}$, if $a(x,\xi)\in S^{-\frac{n}{r}(1-\varrho)}_{\varrho,\delta}$, then
\begin{equation*}\label{weighted2}
\|T_{a}u\|_{L^{p}_{\omega}}\leq C\|u\|_{L^{p}_{\omega}}.
\end{equation*}
if $a\in S^{-\frac{n}{r}(1-\varrho)-\frac{n}{2}\max\{\delta-\varrho,0\}}_{\varrho,\delta}$, then
\begin{equation*}\label{weighted2}
\|T^{*}_{a}u\|_{L^{p}_{\omega}}\leq C\|u\|_{L^{p}_{\omega}}.
\end{equation*}
\end{thm}

\begin{thm}\label{weight3}
Let $0<\varrho\leq1$, $0\leq\delta<1$. For any $1<p<\infty$ and $\omega\in A_{p}$, if $a\in S^{-n(1-\varrho)}_{\varrho,\delta}$, then
\begin{equation*}\label{weighted3}
\|T_{a}u\|_{L^{p}_{\omega}}\leq C\|u\|_{L^{p}_{\omega}}~and~\|T^{*}_{a}u\|_{L^{p}_{\omega}}\leq C\|u\|_{L^{p}_{\omega}}.
\end{equation*}
For $p=1$ and $\omega\in A_{1}$, if $a\in S^{-n(1-\varrho)}_{\varrho,\delta}$, then
\begin{equation*}\label{weighted3}
\|T_{a}u\|_{L^{1,\infty}_{\omega}}\leq C\|u\|_{L^{1}_{\omega}}~and~\|T^{*}_{a}u\|_{L^{1,\infty}_{\omega}}\leq C\|u\|_{L^{1}_{\omega}}.
\end{equation*}
\end{thm}

\begin{thm}\label{weight1}
Let $0<\varrho<1$. For any $1<p<\infty$ and $\omega\in A_{p}$, if $a\in L^{\infty}S^{-n(1-\varrho)}_{\varrho}$, then
\begin{equation*}\label{weighted1}
\|T^{*}_{a}u\|_{L^{p}_{\omega}}\leq C\|u\|_{L^{p}_{\omega}}.
\end{equation*}
For $p=1$ and $\omega\in A_{1}$, if $a\in L^{\infty}S^{-n(1-\varrho)}_{\varrho}$, then
\begin{equation*}\label{weighted1}
\|T^{*}_{a}u\|_{L^{1,\infty}_{\omega}}\leq C\|u\|_{L^{1}_{\omega}}.
\end{equation*}
\end{thm}
The main contribution of these theorems, besides getting the weighted $L^{p}$ boundedness of $T^{*}_{a}$, is extending the range of $\varrho,\delta$ to general case. Especially, the case $p=r=1$ is considered as well. Here, we would like to highlight potential directions for further research, such as extending the study from $L^{p}$ spaces to Morrey spaces. For progress on Calder\'{o}n-Zygmund operators (a class of PDOs) in Morrey spaces, we refer the reader to \cite{Ragusa1, Ragusa2} and the references therein.

Another topic of this paper is to investigate some properties of pseudo-differential operator $T_{a}$ and its dual operators $T^{*}_{a}$ acting on Hardy spaces $H^{p}(\mathbb{R}^{n})$, where $0<p\leq1$. The first property is $(H^{p},H^{p})$ continuity, which can go back to \'{A}lvarez and Milman \cite{Alvarez1,Alvarez2}. They introduce strongly singular Calder\'{o}n-Zygmund operators $T$ and prove the operators $T$ satisfying $T^{*}(1)=0$ acts continuously on $H^{p}(\mathbb{R}^{n})$ for $p_{0}<p\leq1$. As an application, they point out that the pseudo-differential operators $T_{a}$ with symbols in $S^{-\frac{n}{2}(1-\varrho)}_{\varrho,\delta}$ are included in strongly singular Calder\'{o}n-Zygmund operators, where $0<\delta\leq\varrho<1$.
Later, \'{A}lvarez and Hounie \cite{Hounie} extend the range of $\varrho$ and $\delta$ to more general case, that is, $0<\varrho\leq1$ and $0\leq\delta<1$, but $a\in S^{-\frac{n}{2}(1-\varrho)-\frac{n}{2}\max\{\delta-\varrho,0\}}_{\varrho,\delta}$.

\begin{thm}[\'{A}lvarez and Hounie \cite{Hounie}]\label{B2}
Let $0<\varrho\leq1,$ $0\leq\delta<1$ and $a(x,\xi)\in S^{m}_{\varrho,\delta}$.
If $$m\leq -\frac{n}{2}(1-\varrho)-\frac{n}{2}\max\{\delta-\varrho,0\}$$
and $T^{*}_{a}(1)=0$ in the sense of $BMO$.
Then $T_{a}$ maps continuously $H^{p}$ into itself for $p_{0}<p\leq1$ where $\frac{1}{p_{0}}=\frac{1}{2}+\frac{\frac{n}{2}(1-\varrho)(1/\varrho+n/2)}{n(1/\varrho-1+\frac{n}{2}(1-\varrho))}$
\end{thm}
The approach to prove this theorem is applying the atomic and molecular characterization of $H^{p}(\mathbb{R}^{m})$. The advantage of this approach is that one only needs to show that $T_{a}a_{Q}$, the image of a $(p,2)$ atom $a_{Q}$, is a suitable molecule. The condition that $T^{*}_{a}(1)=0$ is used only to provide the cancellation condition of the molecule, that is, $\int_{\mathbb{R}^{n}}T_{a}a_{Q}(x)dx=0$, at cost of restricting the range of $p$ into $p_{0}<p\leq1$. So, the higher degree of cancellation, namely,
\begin{eqnarray}\label{c}
T^{*}_{a}(x^{\alpha})=0,\quad for\quad|\alpha|\leq[n(\frac{1}{p}-1)],
\end{eqnarray}
is required to extend for $p$ below $p_{0}$. Here and below, $[x]$ indicates the integer part of $x$. See \cite{Frazier,Hart,Torres} for the case of Calder\'{o}n-Zygmund operators.
Notice that (\ref{c}) is used only to provide $\int_{\mathbb{R}^{n}}x^{\alpha}T_{a}a_{Q}(x)dx=0$ for $|\alpha|\leq[n(\frac{1}{p}-1)]$.
So, we use the following condition instead of (\ref{c}) in this paper:
\begin{defn}
Let $0<p\leq1$, $t\in \mathbb{N}^{+}\cup\{0\}$, $T$ be a operator and $L^{2}_{c,t}(\mathbb{R}^{n})$ denote the set of functions in $L^{2}_{c}(\mathbb{R}^{n})$ such that$\int_{\mathbb{R}^{n}}x^{\beta}f(x)dx=0$ for $|\beta|\leq t$. If $f\in L^{2}_{c,t}(\mathbb{R}^{n})$, then
\begin{eqnarray}\label{cb}
\int_{\mathbb{R}^{n}}x^{\alpha}Tf(x)dx=0,\quad for\quad|\alpha|\leq[n(\frac{1}{p}-1)].
\end{eqnarray}
Here, $L^{2}_{c}(\mathbb{R}^{n})$ denotes the set of functions in $L^{2}(\mathbb{R}^{n})$  with compact
support.
\end{defn}
As we known, for the atomic decomposition of an element of $H^{p}(\mathbb{R}^{n})$, one can always choose $(p,2)$ atoms with an number of additional vanishing moments that is known as $(p,2,t)$ atoms with $t\geq[n(\frac{1}{p}-1)]$ (see\cite{Stein}). Clearly, if $f$ is a $(p,2,t)$ atom, then $f\in L^{2}_{c,t}(\mathbb{R}^{n})$ with $t\geq[n(\frac{1}{p}-1)]$. Moreover, the proof of Proposition \ref{P2} below implies that (\ref{cb}) for both $T_{a}$ and $T^{*}_{a}$ is well defined, where the symbol $a$ is given in Theorem \ref{TH6}.

\begin{thm}\label{TH6}
Let $0<p<1$, $0\leq\varrho\leq1$ and $0\leq\delta<1$.
\begin{enumerate}
  \item If $T^{*}_{a}$ defined as (\ref{df*}) satisfies condition (\ref{cb}) and $a\in S^{-n(1-\varrho)(\frac{1}{p}-\frac{1}{2})}_{\varrho,\delta}$. Then the operator $T^{*}_{a}$ is bounded on $H^{p}(\mathbb{R}^{n})$.
  \item If $T_{a}$ defined as (\ref{df}) satisfies condition (\ref{cb}) and $a\in S^{-n(1-\varrho)(\frac{1}{p}-\frac{1}{2})-\frac{n}{2}\max(0,\delta-\varrho)}_{\varrho,\delta}$. Then the operator $T_{a}$ is bounded on $H^{p}(\mathbb{R}^{n})$.
\end{enumerate}
\end{thm}

Compared with Theorem \ref{B2}, Theorem \ref{TH6} extend $p$ below $p_{0}$ and improve the range of $m$.

The second property investigated in this paper is $(H^{p},L^{p})$ continuity of pseudo-differential operators, which can go back to Fefferman and Stein\cite{Stein1} and Coifman and Meyer \cite{Coifman} for $p=1$, which is extended to the case $0<p\leq1$ by P\"{a}iv\"{a}rinta and Somersalo \cite{Somersalo}.
\begin{thm}[P\"{a}iv\"{a}rinta and Somersalo \cite{Somersalo}]\label{B1}
Let $0<p\leq1$ and $0\leq\delta\leq\varrho<1$.
If $a\in S^{-n(1-\varrho)(\frac{1}{p}-\frac{1}{2})}_{\varrho,\delta}$. Then the operators $T_{a}$ defined as (\ref{df}) is bounded from $H^{p}(\mathbb{R}^{n})$ to $L^{p}(\mathbb{R}^{n}).$
\end{thm}
Actually, P\"{a}iv\"{a}rinta and Somersalo \cite{Somersalo} get that $T_{a}$ is continuously $h_{p}$ into itself. Here $h_{p}$ denotes the local Hardy spaces introduced by Goldberg \cite{Goldberg}. We also refer to \cite{Park1,Park2} for the extension to Triebel-Lizorkin spaces that coincident with the local Hardy spaces for some special index. Theorem \ref{B1} holds because of the fact $H^{p}\subset h^{p}\subset L^{p}$ for $0<p<\infty$.
As we see, the case $0\leq\varrho<\delta<1$ is not considered in Theorem \ref{B1}. And this case is considered by \'{A}lvarez and Hounie \cite{Hounie} later.
\begin{thm}[\'{A}lvarez and Hounie \cite{Hounie}]\label{B3}
Let $0<\varrho\leq1$, $0\leq\delta<1$ and $p_{0}$ given as Theorem \ref{B2} $($it is understood that for $\varrho=1$, $p_{0}=n/(n+1)$$)$.
If $a\in S^{-\frac{n}{2}(1-\varrho)-\frac{n}{2}\max(0,\delta-\varrho)}_{\varrho,\delta}$. Then the operators $T_{a}$ defined as (\ref{df}) is bounded from $H^{p}(\mathbb{R}^{n})$ to $L^{p}(\mathbb{R}^{n})$ for $p_{0}\leq p\leq1$, when $0<\varrho<1$, and for $p_{0}<p\leq1$, when $\varrho=1$.
\end{thm}
Compared with Theorem \ref{B1}, Theorem \ref{B3} relaxes the range of $\varrho,\delta$, but put a restriction on $p$ and the order of $T_{a}$. Both of them is not contain the case $\varrho=0$, $0<\delta<1$. In this paper, we prove
\begin{thm}\label{TH1}
Let $0<p\leq1$, $0\leq\varrho\leq1$ and $0\leq\delta<1$.
\begin{enumerate}
  \item If $a\in S^{-n(1-\varrho)(\frac{1}{p}-\frac{1}{2})}_{\varrho,\delta}$. Then the operators $T^{*}_{a}$ defined as (\ref{df*}) is bounded from $H^{p}(\mathbb{R}^{n})$ to $L^{p}(\mathbb{R}^{n}).$
  \item If $a\in S^{-n(1-\varrho)(\frac{1}{p}-\frac{1}{2})-\frac{n}{2}\max(0,\delta-\varrho)}_{\varrho,\delta}$. Then the operators $T_{a}$ defined as (\ref{df}) is bounded from $H^{p}(\mathbb{R}^{n})$ to $L^{p}(\mathbb{R}^{n}).$
\end{enumerate}
\end{thm}

\section{The proof of pointwise estimate for the sharp maximal function}
Let \begin{eqnarray}\label{K}
K(x,y)
=\frac{1}{(2\pi)^n}\int_{\mathbb{R}^n} e^{ i \langle x-y,\xi\rangle}a(x,\xi)d\xi
~{\rm and}~
K^{*}(x,y)
=\frac{1}{(2\pi)^n}\int_{\mathbb{R}^n} e^{ i \langle x-y,\xi\rangle}a(y,\xi)d\xi.
\end{eqnarray}
Then $T_{a}$ and $T^{*}_{a}$ can be written as
\begin{eqnarray}
T_{a}u(x)
=\int_{\mathbb{R}^n}K(x,y) u(y)dy \quad
{\rm and}\quad
T^{*}_{a}u(x)
=\int_{\mathbb{R}^n}K^{*}(x,y) u(y)dy
\end{eqnarray}
respectively. Now we introduce the standard Littlewood-Paley partition of unity. Let $C>1$ be a constant. Set
$E_{-1}=\{\xi:|\xi|\leq 2C\}$, $E_{j}=\{\xi:C^{-1}2^{j}\geq|\xi|\leq C2^{j+1}\}$, $j=0,1,2,\cdots$.

\begin{lem}\label{L0}
There exist $\psi_{-1}(\xi),\psi(\xi)\in C^{\infty}_{0}$, such that
\begin{enumerate}
  \item $\supp\psi\subset E_{0}$, $\supp\psi_{-1}\subset E_{-1};$
  \item $0\leq\psi\leq1$, $0\leq\psi_{-1}\leq1;$
  \item $\psi_{-1}(\xi)+\sum\limits^{\infty}_{j=1}\psi(2^{-j}\xi)=1.$
\end{enumerate}
\end{lem}
By Lemma \ref{L0}, the symbol $a(x,\xi)$ can been written as
$$a(x,\xi)=a(x,\xi)\big(\psi_{-1}(\xi)+\sum\limits^{\infty}_{j=1}\psi(2^{-j}\xi)\big)=:\sum\limits^{\infty}_{j=0}a_{j}(x,\xi).$$
Consequently, the operator $T_{a}$ and $T^{*}_{a}$ can been decomposed as
\begin{eqnarray}\label{de}
T_{a}u(x)=\sum\limits_{j=0}^{\infty}T_{j}u(x)
~
 \mathrm{and}~
T^{*}_{a}u(x)=\sum\limits_{j=0}^{\infty}T^{*}_{j}u(x),
\end{eqnarray}
respectively, where
\begin{eqnarray}
T_{j}u(x)
=\int_{\mathbb{R}^n}K_{j}(x,y) u(y)dy \quad
{\rm with}\quad
K_{j}(x,y)
=\frac{1}{(2\pi)^n}\int_{\mathbb{R}^n} e^{ i \langle x-y,\xi\rangle}a_{j}(x,\xi)d\xi
\end{eqnarray}
\begin{eqnarray}
T^{*}_{j}u(x)
=\int_{\mathbb{R}^n}K^{*}_{j}(x,y) u(y)dy
\quad
{\rm with}\quad
K^{*}_{j}(x,y)
=\frac{1}{(2\pi)^n}\int_{\mathbb{R}^n} e^{ i \langle x-y,\xi\rangle}a_{j}(y,\xi)d\xi
\end{eqnarray}

\begin{lem}\label{LL}
Let $0\leq\varrho\leq1,$ $0\leq\delta<1$ and $a(x,\xi)\in S^{m}_{\varrho,\delta}$.
If $1<p\leq2\leq q<\infty$ and $$m\leq -n(\frac{1}{p}-\frac{1}{q})-\frac{n}{2}\max\{\delta-\varrho,0\},$$
then
$$\|T_{a}u\|_{L^{q}}\lesssim\|u\|_{L^{p}}~\mathrm{and}~\|T^{*}_{a}u\|_{L^{q}}\lesssim\|u\|_{L^{p}}.$$
\end{lem}
By Hardy-Littlewood-Sobolev estimate and $L^{2}$-estimate for pseudo-differential operators, \'{A}lvarez and Hounie \cite{Hounie} proved the first inequality in the case of $0<\varrho\leq1$. The case $\varrho=0$ and the second inequality can been gotten by the same way.

\begin{lem}\label{L8}
Let $Q(x_{0},l)$ be a fixed cube with side length $l<1$. Suppose$0\leq\varrho\leq1,$ $0\leq\delta<1$ and  $1<p\leq2$. For any positive integer $j$ satisfying $2^{j}l<1$, if  $a(x,\xi)\in S^{-\frac{n}{p}(1-\varrho)}_{\varrho,\delta}$, then
\begin{eqnarray}\label{E012}
\int_{\mathbb{R}^{n}}|u(y)||K_{j}(x,y)-K_{j}(z,y)|dy\lesssim 2^{j}lM_{p}u(x_{0}),\quad \forall x,z\in Q(x_{0},l).
\end{eqnarray}

if  $a(x,\xi)\in S^{-\frac{n}{p}(1-\varrho)-\frac{n}{2}\max\{\delta-\varrho,0\}}_{\varrho,\delta}$, then
\begin{eqnarray}\label{E0120}
\int_{\mathbb{R}^{n}}|u(y)||K^{*}_{j}(x,y)-K^{*}_{j}(z,y)|dy\lesssim 2^{j}lM_{p}u(x_{0}),\quad \forall x,z\in Q(x_{0},l).
\end{eqnarray}
\end{lem}

\begin{proof}
The idea behind the proof of (\ref{E012}) is standard which could be found in \cite{Chanillo}. So we omit it here. However, to prove (\ref{E0120}), this method has to be modified since the Parseval's identity can not be used directly. So we list the details here. First, integrand of left side of (\ref{E0120}) can be bounded by
\begin{eqnarray}\label{E12}
\int_{\mathbb{R}^{n}}|u(y)||\int_{\mathbb{R}^n}\big(e^{i\langle x-y,\xi\rangle}-e^{ i \langle z-y,\xi\rangle}\big)a_{j}(y,\xi)d\xi|dy.
\end{eqnarray}
Break up this integrand as follows
\begin{eqnarray*}
\int_{|y-x_{0}|\leq2^{-j\varrho+1}}
+
\int_{|y-x_{0}|>2^{-j\varrho+1}}
\end{eqnarray*}
H\"{o}lder's inequality show that the first term is bounded by
\begin{eqnarray}\label{b2}
\big(\int_{|y-x_{0}|\leq2^{-j\varrho+1}}|u(y)|^{p}dy\big)^{\frac{1}{p}}
\big(\int_{\mathbb{R}^{n}}|\int_{\mathbb{R}^n} e^{ i \langle \tilde{x}-y,\xi\rangle}a_{j}(y,\xi)(x-z)\cdot\xi d\xi|^{p'}dy\big)^{\frac{1}{p'}},
\end{eqnarray}
where $\tilde{x}$ denotes some point between $x$ and $z$. For any fixed $x$ and $z$, let $b_{j}(y,\xi)=a_j(\tilde{x}-y,\xi)|\xi|^{\frac{n}{p}(1-\varrho)-n(\frac{1}{2}-\frac{1}{p'})}$ and $\widehat{g_{j}(\xi)}=|\xi|^{-\frac{n}{p}(1-\varrho)+n(\frac{1}{2}-\frac{1}{p'})}\chi_{j}(\xi)(x-z)\cdot\xi $. Then we can write
\begin{eqnarray*}
\int_{\mathbb{R}^n} e^{ i \langle y,\xi\rangle}a_{j}(\tilde{x}-y,\xi)(x-z)\cdot\xi d\xi=T_{b_{j}}g_{j}(y)
\end{eqnarray*}
Notice that $b_{j}\in S^{-n(\frac{1}{2}-\frac{1}{p'})-\frac{n}{2}\max\{\delta-\varrho,0\}}_{\varrho,\delta}$, we have by Lemma \ref{LL}
$$\|T^{*}_{b_{j}}g_{j}\|_{L^{p'}}\lesssim \|g_{j}\|_{L^{2}}=\|\hat{g_{j}}\|_{L^{2}}.$$
Therefor, (\ref{b2}) is bounded by
\begin{eqnarray*}
2^{j}lM_{p}u(x_{0}).
\end{eqnarray*}

By H\"{o}lder's inequality, integrating by parts and the fact  $|y-x_{0}|\sim|y-x|$ that follows from $2^{j}l<1$, $x\in Q(x_{0},l)$ and $|y-x_{0}|>2^{-j\varrho+1}$, the second term is bounded by
\begin{eqnarray}\label{b3}
&&\big(\int_{|y-x_{0}|>2^{-j\varrho+1}}\frac{|u(y)|^{p}}{|y-x_{0}|^{pN}}dy\big)^{\frac{1}{p}}\nonumber\\
&&\quad\times\sum\limits_{|\alpha|=N}\big(\int_{\mathbb{R}^n}|\int_{\mathbb{R}^n} e^{ i \langle \tilde{x}-y,\xi\rangle}\partial^{\alpha}_{\xi}\big(a_{j}(y,\xi)(x-z)\cdot\xi\big) d\xi|^{p'}dy\big)^{\frac{1}{p'}}.
\end{eqnarray}
For any fixed $x$ and $z$, let $\tilde{b}_{j}(y,\xi)=\partial^{\alpha}_{\xi}\big(a_j(y,\xi)(x-z)\cdot\xi\big)|\xi|^{\frac{n}{p}(1-\varrho)-n(\frac{1}{2}-\frac{1}{p'})+\varrho |\alpha|}$ and $\widehat{\tilde{g}_{j}(\xi)}=|\xi|^{-\frac{n}{p}(1-\varrho)+n(\frac{1}{2}-\frac{1}{p'})-\varrho|\alpha|}\chi_{j}(\xi)$. Then we can write
\begin{eqnarray*}
\int_{\mathbb{R}^n} e^{ i \langle y,\xi\rangle}\partial^{\alpha}_{\xi}\big(a_{j}(\tilde{x}-y,\xi)(x-z)\cdot\xi\big) d\xi
=T^{*}_{\tilde{b}_{j}}\tilde{g_{j}}(y)
\end{eqnarray*}
Clearly, $\tilde{b}_{j}\in S^{-n(\frac{1}{2}-\frac{1}{p'})-\frac{n}{2}\max\{\delta-\varrho,0\}}_{\varrho,\delta}$ with bounds $\lesssim 2^{j}l$. Moreover we have by Lemma \ref{LL}
$$\|T^{*}_{\tilde{b}_{j}}\tilde{g_{j}}\|_{L^{p'}}\lesssim 2^{j}l\|\tilde{g_{j}}\|_{L^{2}}=2^{j}l\|\hat{\tilde{g_{j}}}\|_{L^{2}}.$$
By simple calculation, we can get (\ref{b3}) is bounded by
\begin{eqnarray*}
&\lesssim&2^{j}lM_{p}u(x_{0}).
\end{eqnarray*}
Thus the following desired estimate can be gotten.
\end{proof}

\begin{lem}\label{L9}
Let $Q(x_{0},l)$ be a fixed cube with side length $l<1$ $0\leq\varrho\leq1$ and $0\leq\delta<1$. For any positive integer $N>\frac{n}{p}$ and any positive integer $j$ with $l^{-1}\leq2^{j}\leq l^{-\frac{1}{\varrho}}$,  if $a\in S^{-\frac{n}{p}(1-\varrho)-\frac{n}{2}\max\{\delta-\varrho,0\}}_{\varrho,\delta}$, then
\begin{eqnarray}\label{u00}
\frac{1}{|Q|}\int_{Q(x_{0},l)}|T_{j}u(x)|dx
&\lesssim&2^{j\frac{n}{2}(\frac{n}{Np}-1)}l^{\frac{n}{2}(\frac{n}{Np}-1)}M_{p}u(x_{0})
\end{eqnarray}
and
\begin{eqnarray}\label{u001}
\frac{1}{|Q|}\int_{Q(x_{0},l)}|T^{*}_{j}u(x)|dx
&\lesssim&2^{j\frac{n}{2}(\frac{n}{Np}-1)}l^{\frac{n}{2}(\frac{n}{Np}-1)}M_{p}u(x_{0}).
\end{eqnarray}
\end{lem}
\begin{rem}\label{R1}
If $\varrho=0$, the condition $l^{-1}\leq2^{j}\leq l^{-\frac{1}{\varrho}}$ is interpreted as $l^{-1}\leq2^{j}$. If $\varrho=1$, this lemma is no use.
\end{rem}

\begin{proof}
Notice that $a(x,\xi) \psi(2^{-j}\xi)\in S^{-n(\frac{1}{p}-\frac{1}{2})-\frac{n}{2}\max\{\delta-\varrho,0\}}_{\varrho, \delta}$ with the bounds $\lesssim2^{-j\frac{n}{p}(1-\varrho)+n(\frac{1}{p}-\frac{1}{2})}$. So $T_{j}$ is bounded from $L^{p}$ to $L^{2}$, see Lemma \ref{LL}. More exactly, we have
$$\|T_{j}u\|_{L^{2}}\lesssim2^{-j\frac{n}{p}(1-\varrho)+n(\frac{1}{p}-\frac{1}{2})}\|u\|_{L^{p}}.$$
Let integral $N$ defined as above and
set
$$T=l^{\frac{n}{2N}}2^{j(\frac{n}{2N}-\varrho)}$$
\begin{eqnarray}\label{ui}
u_{1}(x)=u(x)\chi_{Q(x_{0},4T)}(x) \quad{\rm and} \quad u_{2}(x)=u(x)-u_{1}(x),
\end{eqnarray}
where $\chi_{Q(x_{0},4T)}(x)$ is the characteristic function of the ball $Q(x_{0},4T).$ Then the left hand of (\ref{u00}) can be bounded by
\begin{eqnarray*}
\int_{Q(x_{0},l)}|T_{j}u_{1}(x)|dx+
\int_{Q(x_{0},l)}|T_{j}u_{2}(x)|dx=:M_{1}+M_{2}.
\end{eqnarray*}
H\"{o}lder's inequality and $(p,2)$-boundedness of $T_{j}$ imply that $M_{1}$ is bounded by
\begin{eqnarray}\label{b11}
l^{\frac{n}{2}}\|T_{j}u_{1}\|_{L^{2}}\nonumber
&\lesssim& 2^{-j\frac{n}{p}(1-\varrho)+n(\frac{1}{p}-\frac{1}{2})}l^{\frac{n}{2}}\|u_{1}\|_{L^{p}}\\
&\lesssim&2^{j\frac{n}{2}(\frac{n}{Np}-1)}l^{\frac{n}{2}(\frac{n}{Np}+1)}M_{p}u(x_{0}).
\end{eqnarray}

For $M_{2}$, noticing that for any $x\in Q(x_{0},l)$ and any $y\in Q^{C}(x_{0},4T)$, we have
$$|y-x|\geq\frac{|y-x_{0}|}{2}.$$
H\"{o}lder's inequality, Integrating by parts  and Parseval's identity give that $|T_{j}u_{2}(x)|$ is bounded by
\begin{eqnarray*}
&&\big(\int_{|y-x_{0}|>4T}\frac{|u(y)|^{p}}{|y-x_{0}|^{pN}}dy\big)^{\frac{1}{p}}
\big(\int_{|y-x_{0}|>4T}|y-x_{0}|^{p'N}|\int_{\mathbb{R}^n} e^{ i \langle x-y,\xi\rangle}a(x,\xi)\psi(2^{-j}\xi) d\xi|^{p'}dy\big)^{\frac{1}{p'}}\nonumber\\
&\lesssim&\big(\int_{|y-x_{0}|>4T}\frac{|u(y)|^{p}}{|y-x_{0}|^{pN}}dy\big)^{\frac{1}{p}}\big(\int_{\mathbb{R}^n} |\partial^{\alpha}_{\xi}a( x,\xi)\psi(2^{-j}\xi)|^{p} d\xi\big)^{\frac{1}{p}}\nonumber\\
&\lesssim&2^{j\frac{n}{2}(\frac{n}{Np}-1)}l^{\frac{n}{2}(\frac{n}{Np}-1)}M_{p}u(x_{0}).
\end{eqnarray*}
 So
\begin{eqnarray}\label{b20}
M_{2}=\int_{Q(x_{0},l)}|T_{j}u_{2}(x)|dx\lesssim
2^{j\frac{n}{2}(\frac{n}{Np}-1)}l^{\frac{n}{2}(\frac{n}{Np}+1)}M_{p}u(x_{0}).
\end{eqnarray}
Thus, the desired estimate (\ref{u00}) follows from (\ref{b11}) and (\ref{b20}).
So we complete the proof.
\end{proof}

\begin{lem}\label{L7}
Let $Q(x_{0},l)$ be a fixed cube with side length $l<1$. Suppose $0<\varrho<\delta<1$, $a\in S^{-\frac{n}{p}(1-\varrho)}_{\varrho,\delta}$. Then for any $1\leq\lambda\leq\frac{1}{\varrho}$, any positive integer $N>\frac{n}{p}$ and any positive integer $j$ with $l^{-\lambda}\leq2^{j}\leq l^{-\frac{1}{\varrho}}$, we have
\begin{eqnarray*}
\frac{1}{|Q|}\int_{Q(x_{0},l)}|T_{j}u(x)|dx
&\lesssim&\big(2^{j\delta}l^{\lambda}+2^{j\frac{n}{2}(\frac{n}{Np}-1)}l^{\frac{n\lambda}{2}(\frac{n}{Np}-1)}\big)M_{p}u(x_{0})
\end{eqnarray*}
\end{lem}
\begin{proof}
If $1<\lambda\leq\frac{1}{\varrho}$, then $l^{\lambda}<l$. Take integer $L$ such that it is the first number no less than  $l^{1-\lambda}$, that is $L-1<l^{1-\lambda}\leq L$. Then there are $L^{n}$ cubes with the same side length $l^{\lambda}$ covering $Q(x_{0},l)$. Moreover, we have
$$Q(x_{0},l)\subset\cup_{i=1}^{L^{n}}Q(x_{i},l^{\lambda})\subset Q(x_{0},2l).$$
Clearly, $L^{n}\leq2^{n}l^{n(1-\lambda)}$. Denote
\begin{eqnarray}\label{0A}
T_{j,i}u(x)
&=&\int_{\mathbb{R}^n} e^{ i \langle x,\xi\rangle}a(x_{i},\xi) \psi(2^{-j}\xi) \hat{u}(\xi)d\xi.
\end{eqnarray}
We write
\begin{eqnarray}\label{A}
&&\frac{1}{|Q|}\int_{Q(x_{0},l)}|T_{j}u(x)|dx\nonumber\\
&\leq&\frac{1}{|Q|}\sum\limits_{i=1}^{L^{n}}\bigg(\int_{Q(x_{i},l^{\lambda})}|T_{j}u(x)-T_{j,i}u(x)|dx+\int_{Q(x_{i},l^{\lambda})}|T_{j,i}u(x)|dx\bigg).
\end{eqnarray}
Now we claim that
\begin{eqnarray}\label{a}
|T_{j}u( x)-T_{j,i}u( x)|\lesssim|x-x_{i}|2^{j\delta}M_{p}u(x_{0}),
\end{eqnarray}
\begin{eqnarray}\label{b}
\int_{Q(x_{i},l^{\lambda})}|T_{j,i}u(x)|dx
&\lesssim&2^{j\frac{n}{2}(\frac{n}{Np}-1)}l^{\frac{n\lambda}{2}(\frac{n}{Np}+1)}M_{p}u(x_{0}).
\end{eqnarray}
Since $L^{n}\leq2^{n}l^{n(1-\lambda)}$, we can get the desired estimate by substituting both (\ref{a}) and (\ref{b}) into (\ref{A}).

Note that $|T_{j}u( x)-T_{j,i}u( x)|$ is bounded by
\begin{eqnarray*}
\int_{\mathbb{R}^n}|u(y)||\int_{\mathbb{R}^n} e^{ i \langle x-y,\xi\rangle}\big(a(x,\xi)-a(x_{i},\xi)\big)\psi(2^{-j}\xi) d\xi|dy.
\end{eqnarray*}
Then, \eqref{a} follows from the same argument as \eqref{E12}.

Now, we prove (\ref{b}).For fixed $x_{i}$, we can see that $a(x_{i},\xi) \psi(2^{-j}\xi)\in S^{-n(\frac{1}{p}-\frac{1}{2})}_{\varrho, 0}$ with the bounds $\lesssim2^{-j\frac{n}{p}(1-\varrho)+n(\frac{1}{p}-\frac{1}{2})}$. So $T_{j,i}$ is bounded from $L^{p}$ to $L^{2}$. More exactly, we have
$$\|T_{j,i}u\|_{L^{2}}\lesssim2^{-j\frac{n}{p}(1-\varrho)+n(\frac{1}{p}-\frac{1}{2})}\|u\|_{L^{p}}.$$ Fix positive integral $N$ large enough and
set
$$T=l^{\frac{n\lambda}{2N}}2^{j(\frac{n}{2N}-\varrho)}$$
\begin{eqnarray}\label{ui}
u_{i,1}(x)=u(x)\chi_{Q(x_{i},4T)}(x) \quad{\rm and} \quad u_{i,2}(x)=u(x)-u_{i,1}(x),
\end{eqnarray}
where $\chi_{Q(x_{i},4T)}(x)$ is the characteristic function of the ball $Q(x_{i},4T).$ Then (\ref{b}) follows from the same argument as \eqref{u00}.

If $\lambda=1$, we define
\begin{eqnarray}
T_{j,0}u(x)
&=&\int_{\mathbb{R}^n} e^{ i \langle x,\xi\rangle}a(x_{0},\xi) \psi(2^{-j}\xi) \hat{u}(\xi)d\xi.
\end{eqnarray}
Then the desired estimate can be got by the same argument as above with $T_{j,i}u$ replaced by $T_{j,0}u$. So we complete the proof.
\end{proof}

We remark that the same result holds for the case $\varrho=0$.  Here, the range of $\lambda$ can be extended to $[1,\infty)$.
However, to make some sums convergent, $\lambda$ has to be confined to a finite range.
\begin{lem}\label{La0}
Let $Q(x_{0},l)$ be a fixed cube with side length $l<1$. Suppose $\varrho= 0$, $0<\delta<1$, $a\in S^{-\frac{n}{p}}_{0,\delta}$, then for any $1\leq\lambda\leq\frac{2}{p(1-\delta)}$, any positive integer $N>\frac{n}{p}$ and any positive integer $j$ with $l^{-\lambda}\leq2^{j}\leq l^{-\frac{2}{p(1-\delta)}}$,
\begin{eqnarray*}
\frac{1}{|Q|}\int_{Q(x_{0},l)}|T_{j}u(x)|dx
&\lesssim&\big(2^{j\delta}l^{\lambda}+2^{j\frac{n}{2}(\frac{n}{Np}-1)}l^{\frac{n\lambda}{2}(\frac{n}{Np}-1)}\big)M_{p}u(x_{0}).
\end{eqnarray*}
\end{lem}

\begin{lem}\label{La00}
Let $Q(x_{0},l)$ be a fixed cube with side length $l<1$. Suppose $1<p\leq2$, $\varrho= 0$, $0\leq\delta<1$, $a\in S^{-\frac{n}{p}}_{0,\delta}$, then for any positive integer $N>\frac{n}{p}$ and any positive integer $j$ with $l^{-\frac{2}{p(1-\delta)}}\leq2^{j}$,
\begin{eqnarray*}
\frac{1}{|Q|}\int_{Q(x_{0},l)}|T_{j}u(x)|dx
&\lesssim&2^{-j\frac{n}{2}(1-\delta)(1-\frac{n}{pN})}l^{-\frac{n}{p}(1-\frac{n}{pN})}M_{p}u(x_{0}).
\end{eqnarray*}
\end{lem}

\begin{proof}
Denote
$$\Gamma=2^{j\frac{n}{pN}(1-\delta)}l^{\frac{n}{pN}}.$$
Set $u_{3}(x)=u(x)\chi_{Q(x_{0},2\Gamma)}(x)$ and $u_{4}(x)=u(x)-u_{3}(x)$.
Then
\begin{eqnarray}\label{Ea}
\frac{1}{|Q|}\int_{Q(x_{0},l)}|T_{j}u(x)|dx
\leq\frac{1}{|Q|}\int_{Q(x_{0},l)}|T_{j}u_{3}(x)|dx+
\frac{1}{|Q|}\int_{Q(x_{0},l)}|T_{j}u_{4}(x)|dx.
\end{eqnarray}

Notice that $a(x,\xi) \psi(2^{-j}\xi)\in S^{-n(\frac{1}{p}-\frac{1}{2})-\frac{n}{2}\delta}_{\varrho,\delta}$ with bounds $\lesssim 2^{-j\frac{n}{2}(1-\delta)}$. H\"{o}lder's inequality and the $L^{p}$-estimate of $T_{j}$ give that
\begin{eqnarray}\label{E10}
\frac{1}{|Q|}\int_{Q}|T_{j}u_{3}(x)|dx
&\lesssim& 2^{-j\frac{n}{2}(1-\delta)}l^{-\frac{n}{p}}\|u_{1}\|_{L^p}\nonumber
\lesssim 2^{-j\frac{n}{2}(1-\delta)}l^{-\frac{n}{p}}\Gamma^{\frac{n}{p}}M_{p}u(x_{0})\\
&=&2^{-j\frac{n}{2}(1-\delta)(1-\frac{n}{pN})}l^{-\frac{n}{p}(1-\frac{n}{pN})}M_{p}u(x_{0}).
\end{eqnarray}

Notice that $\Gamma>l$. We have $|y-x|\sim|y-x_{0}|$ for $\forall x\in Q(x_{0},l)$ and $\forall y\in Q^{C}(x_{0},2\Gamma)$. So direct computations show that
\begin{eqnarray*}
|T_{j}u_{4}(x)|
&\leq&\int_{|y-x_{0}|\geq 2\Gamma}|K_{j}(x,x-y)||u(y)|dy
\lesssim \Gamma^{(\frac{n}{p}-N)} M_{p}u(x_{0})\\
&=&2^{-j\frac{n}{2}(1-\delta)(1-\frac{n}{pN})}l^{-\frac{n}{p}(1-\frac{n}{pN})}M_{p}u(x_{0}),
\end{eqnarray*}
which implies that
\begin{eqnarray}\label{E11}
\frac{1}{|Q|}\int_{Q(x_{0},l)}|T_{j}u_{4}(x)|dx
\lesssim 2^{-j\frac{n}{2}(1-\delta)(1-\frac{n}{pN})}l^{-\frac{n}{p}(1-\frac{n}{pN})}M_{p}u(x_{0}).
\end{eqnarray}
Clearly, the desired estimate follows from \eqref{Ea}, \eqref{E10} and \eqref{E11}.
\end{proof}

Taking $\Gamma=l$ in the proof Lemma \ref{La00}, we can get a similar result for $\varrho>0$ with the same argument as above.
\begin{lem}\label{L70}
Let $Q(x_{0},l)$ be a fixed cube with side length $l<1$. Suppose  $0<\varrho\leq1$, $0\leq\delta<1$ and $1<p\leq2$. For any positive integer $N>\frac{n}{p}$ and any positive integer $j$ with $l^{-\frac{1}{\varrho}}\leq2^{j}$,
if $a\in S^{-\frac{n}{p}(1-\varrho)}_{\varrho,\delta}$ then
\begin{eqnarray*}
\frac{1}{|Q|}\int_{Q(x_{0},l)}|T_{j}u(x)|dx
&\lesssim&\big(2^{-j(\frac{n}{2}(1-\varrho)-\frac{n}{2}\max\{\delta-\varrho,0\})}+2^{-j\varrho(\frac{n}{p}-N)}l^{\frac{n}{p}-N)}\big)M_{p}u(x_{0})
\end{eqnarray*}
and
\begin{eqnarray*}
\frac{1}{|Q|}\int_{Q(x_{0},l)}|T^{*}_{j}u(x)|dx
&\lesssim&\big(2^{-j(\frac{n}{2}(1-\varrho)-\frac{n}{2}\max\{\delta-\varrho,0\})}+2^{-j\varrho(\frac{n}{p}-N)}l^{\frac{n}{p}-N)}\big)M_{p}u(x_{0});
\end{eqnarray*}
\end{lem}

\begin{proof}[Proof of Theorem \ref{TH2}]

Without loss of generality, we assume that the symbol $a(x,\xi)$ vanishes for $|\xi|\leq 1$. Let $Q=Q(x_{0},l)$ denote the cube centered at $x_{0}$ with the side length $l.$ For any fixed cube $Q$, we are going to prove that
\begin{eqnarray}\label{E0}
\frac{1}{|Q|}\int_{Q}|T_{a}u(x)-C_{Q}|dx\leq C M_{p}u(x_{0}),
\end{eqnarray}
where $C_{Q}=\frac{1}{|Q|}\int_{Q}T_{a}u(y)dy$. The proof is trivial for $l\geq1$, we omit it here. We put our eyes on $0<l<1$. Note that the left hand of (\ref{E0}) can be controlled by
\begin{eqnarray}\label{E1}
\frac{1}{|Q|^{2}}\int_{Q}\int_{Q}|T_{a}u(x)-T_{a}u(y)|dydx.
\end{eqnarray}
We compose the operator $T_{a}$ as (\ref{de}), then estimate (\ref{E1}) by
\begin{eqnarray}\label{E13}
\sum\limits_{1<2^{j}\leq l^{-1}}\frac{1}{|Q|^{2}}\int_{Q}\int_{Q}|T_{j}u(x)-T_{j}u(z)|dzdx
&+&\sum\limits_{l^{-1}<2^{j}}\frac{2}{|Q|}\int_{Q}|T_{j}u(x)|dx.
\end{eqnarray}
Lemma (\ref{L8}) implies that
\begin{eqnarray*}
|T_{j}u(x)-T_{j}u(z)|\leq  \int_{\mathbb{R}^{n}}|u(y)||K_{j}(x,y)-K_{j}(z,y)|dy\leq C 2^{j}|x-z|M_{p}u(x_{0}).
\end{eqnarray*}
So the first term in (\ref{E13}) is bounded by
\begin{eqnarray}
 M_{p}u(x_{0})l\sum\limits_{1<2^{j}\leq l^{-1}}2^{j}\lesssim M_{p}u(x_{0}).
\end{eqnarray}
Next we claim that the second term in (\ref{E13}) can be controlled by $M_{p}u(x_{0})$ as well.

Case 1 $0\leq\delta\leq\varrho\leq1$, $\delta\neq1$
If $\varrho=0$, then by Lemma \ref{L9} and Remark \ref{R1} the second term in (\ref{E13}) can be bounded by
\begin{eqnarray*}
\sum\limits_{l^{-1}<2^{j}}2^{j\frac{n}{2}(\frac{n}{Np}-1)}l^{\frac{n}{2}(\frac{n}{Np}-1)}M_{p}u(x_{0})
\lesssim M_{p}u(x_{0}).
\end{eqnarray*}
If $\varrho\neq0$, we break up this sum as follows
\begin{eqnarray}\label{E12}
\sum\limits_{l^{-1}<2^{j}\leq l^{-\frac{1}{\varrho}}}\frac{2}{|Q|}\int_{Q}|T_{j}u(x)|dx
+\sum\limits_{l^{-\frac{1}{\varrho}}<2^{j}}\frac{2}{|Q|}\int_{Q}|T_{j}u(x)|dx.
\end{eqnarray}
Then Lemma \ref{L9} and Lemma \ref{L70} imply that they can be controlled by
\begin{eqnarray*}
&&\sum\limits_{l^{-1}<2^{j}\leq l^{-\frac{1}{\varrho}}}2^{j\frac{n}{2}(\frac{n}{Np}-1)}l^{\frac{n}{2}(\frac{n}{Np}-1)}M_{p}u(x_{0})\\
&+&\sum\limits_{l^{-\frac{1}{\varrho}}<2^{j}}
\big(2^{-j(\frac{n}{2}(1-\varrho)-\frac{n}{2}\max\{\delta-\varrho,0\})}+2^{-j\varrho(\frac{n}{p}-N)}l^{\frac{n}{p}-N)}\big)M_{p}u(x_{0})
\lesssim M_{p}u(x_{0}).
\end{eqnarray*}
Case 2 $0\leq\varrho<\delta<1$ If $\varrho\neq0$, we break up this sum as (\ref{E12}) as well. By Lemma \ref{L70}, the second term in (\ref{E12}) can be controlled by $M_{p}u(x_{0})$. For the first term in in (\ref{E12}, we write
\begin{eqnarray*}
\sum\limits_{l^{-1}<2^{j}\leq l^{-\frac{1}{\varrho}}}\frac{2}{|Q|}\int_{Q}|T_{j}u(x)|dx
&=&\big(\sum\limits_{l^{-1}<2^{j}\leq l^{-\frac{1}{\delta}}}
+\sum\limits_{l^{-\frac{1}{\delta}}<2^{j}\leq l^{-\frac{1}{\delta^{2}}}}+...
+\sum\limits_{l^{-\frac{1}{\delta^{k-1}}}<2^{j}\leq l^{-\frac{1}{\delta^{k}}}}\\
&+&...
+\sum\limits_{l^{-\frac{1}{\delta^{\gamma-1}}}<2^{j}\leq \min\{l^{-\frac{1}{\varrho}},l^{-\frac{1}{\delta^{\gamma}}}\}}\big)
\frac{1}{|Q|}\int_{Q(x_{0},l)}|T_{j}u(x)|dx,
\end{eqnarray*}
where $\gamma$ is the first positive integer such that $\frac{1}{\delta^{\gamma}}\geq \frac{1}{\varrho}$.
Then take $\lambda=\frac{1}{\delta^{k}}$, $k=0,1,...,\gamma-1$ in Lemma \ref{L7} respectively, we can see that each sum above is bounded by $M_{p}u(x_{0})$. Therefore we have
\begin{eqnarray}
\sum\limits_{l^{-1}<2^{j}\leq l^{-\frac{1}{\varrho}}}\frac{2}{|Q|}\int_{Q}|T_{j}u(x)|dx
\leq C_{\gamma}M_{p}u(x_{0}).
\end{eqnarray}

If $\varrho=0$, we break up this sum as follows
\begin{eqnarray}
\sum\limits_{l^{-1}<2^{j}\leq l^{-\frac{2}{p(1-\delta)}}}\frac{2}{|Q|}\int_{Q}|T_{j}u(x)|dx
+\sum\limits_{l^{-\frac{2}{p(1-\delta)}}<2^{j}}\frac{2}{|Q|}\int_{Q}|T_{j}u(x)|dx.
\end{eqnarray}
Applying Lemma \ref{La0} and Lemma \ref{La00} instead of Lemma \ref{L70} and Lemma \ref{L7}, we can get the desired estimate by the same argument as above. So the proof is finished.

\end{proof}

Next, we started to prepare for proving the case $p=1$, that is, Theorem \ref{TH3} and Theorem \ref{TH5}.

\begin{lem}\label{L3}
Let $Q(x_{0},l)$ be a fixed cube with side length $l<1$. Suppose $0\leq\varrho\leq1$ and $0\leq\delta<1$. For any positive integer $j$ satisfying $2^{j}l<1$, if  $a(x,\xi)\in S^{-n(1-\varrho)}_{\varrho,\delta}$, then
\begin{eqnarray}\label{E2}
\int_{\mathbb{R}^{n}}|u(y)||K_{j}(x,y)-K_{j}(z,y)|dy\lesssim 2^{j}lMu(x_{0}),\quad \forall x,z\in Q(x_{0},l)
\end{eqnarray}
and
\begin{eqnarray}\label{E4}
\int_{\mathbb{R}^{n}}|u(y)||K^{*}_{j}(x,y)-K^{*}_{j}(z,y)|dy\lesssim 2^{j}lMu(x_{0}),\quad \forall x,z\in Q(x_{0},l).
\end{eqnarray}
\end{lem}

\begin{proof}
The proof of (\ref{E2})  and (\ref{E4}) is standard. We only show a outline of proving (\ref{E4}).
Integrand of left side of (\ref{E4}) can be bounded by
\begin{eqnarray}\label{E40}
\int_{\mathbb{R}^{n}}|u(y)||\int_{\mathbb{R}^n}\big(e^{i\langle x-y,\xi\rangle}-e^{ i \langle z-y,\xi\rangle}\big)a_{j}(y,\xi)d\xi|dy.
\end{eqnarray}
Break up this integrand as follows
\begin{eqnarray*}
\int_{|y-x_{0}|\leq2^{-j\varrho+1}}
+
\int_{|y-x_{0}|>2^{-j\varrho+1}}.
\end{eqnarray*}
A direct calculation gives the first term is bounded by $2^{j}lMu(x_{0})$, and integration by parts with respect to the variable $\xi$ yields that the second term has the same bound.
Thus the proof is completed.
\end{proof}

\begin{rem}\label{R41}
Notice that the smooth of variable $y$ in $a(y.\xi)$ is not used in the proof of (\ref{E4}). So, it can be get in a relaxed condition. More exactly, (\ref{E4}) can been gotten under condition $a(y,\xi)\in L^{\infty}S^{-n(1-\varrho)}_{\varrho}$.
\end{rem}

\begin{lem}\label{La4}
Let $Q(x_{0},l)$ be a fixed cube with side length $l<1$. Suppose $0<\varrho\leq1$ and $0\leq\delta<1$. For any positive integer $j$ with $l^{-1}\leq2^{j}\leq l^{-\frac{1}{\varrho}}$, if
$a\in S^{-n(1-\varrho)-\frac{n}{2}\max\{\delta-\varrho,0\}}_{\varrho,\delta}$
then
\begin{eqnarray}\label{E5}
\frac{1}{|Q(x_{0},l)|}\int_{Q(x_{0},l)}|T_{j}f(x)|dx
&\lesssim&2^{-j\frac{n}{2}}l^{-\frac{n}{2}}Mf(x_{0});
\end{eqnarray}
if $a\in S^{-n(1-\varrho)}_{\varrho,\delta}$
then
\begin{eqnarray}\label{E7}
\frac{1}{|Q(x_{0},l)|}\int_{Q(x_{0},l)}|T^{*}_{j}f(x)|dx
&\lesssim&2^{-j\frac{n}{2}}l^{-\frac{n}{2}}Mf(x_{0}).
\end{eqnarray}
\end{lem}

\begin{proof}
We prove (\ref{E5}) first. H\"{o}lder's inequality and Minkowski's inequality implies that the left hand in (\ref{E5}) can be bounded by
\begin{eqnarray*}
l^{-\frac{n}{2}}\int_{\mathbb{R}^{n}}\big(\int_{|x-x_{0}|<l}|K_{j}(x,y)|^{2}dx\big)^{\frac{1}{2}}|f(y)|dy.
\end{eqnarray*}
So, it suffices to show
\begin{eqnarray*}
\int_{\mathbb{R}^{n}}\big(\int_{|x-x_{0}|<l}|K_{j}(x,y)|^{2}dx\big)^{\frac{1}{2}}|f(y)|dy
&\lesssim&2^{-j\frac{n}{2}}Mf(x_{0}).
\end{eqnarray*}
Break up the integral with respect to the variable $y$ as follows
\begin{eqnarray}\label{E01}
\int_{|y-x_{0}|\leq2^{-j\varrho+1}}+\int_{|y-x_{0}|>2^{-j\varrho+1}}.
\end{eqnarray}
Let $c_{j}(x,\xi)=a_j(x,\xi)|\xi|^{n(1-\varrho)}$ and $\widehat{h_{j}(\xi)}=|\xi|^{-n(1-\varrho)}\chi_{j}(\xi)$. Then we can write
\begin{eqnarray*}
K_{j}(x,y)=\int_{{\mathbb{R}^n}}e^{ i\langle x-y, \xi\rangle}a_j(x,\xi)d\xi
=\int_{{\mathbb{R}^n}}e^{ i\langle x-y, \xi\rangle}c_j(x,\xi)\hat{h_{j}(\xi)}d\xi
=T_{c_{j}}h_{j}(x-y)
\end{eqnarray*}
So, the first term in (\ref{E01}) can be wrote as
\begin{eqnarray*}
\int_{|y-x_{0}|\leq2^{-j\varrho+1}}\big(\int_{\mathbb{R}^{n}}|T_{c_{j}}h_{j}(x-y)|^{2}dx\big)^{\frac{1}{2}}|f(y)|dy
\end{eqnarray*}
Notice $c_{j}\in S^{-\frac{n}{2}\max\{\delta-\varrho,0\}}_{\varrho,\delta}$. Moreover $T_{c_{j}}$ is bounded on $L^{2}$. So it can be bounded by
\begin{eqnarray*}
\int_{|y-x_{0}|\leq2^{-j\varrho+1}}|f(y)|dy\big(\int_{\mathbb{R}^{n}}|h_{j}(\xi)|^{2}d\xi\big)^{\frac{1}{2}}
\leq 2^{-j\frac{n}{2}}Mf(x_{0}).
\end{eqnarray*}

Now we estimate the second term in (\ref{E01}). For positive integer $N>n$, denote $\tilde{c}_{j}(x,\xi)=\partial^{N}_{\xi}a_j(x,\xi)|\xi|^{n(1-\varrho)+\varrho N}$ and $\widehat{\tilde{h}_{j}(\xi)}=|\xi|^{-n(1-\varrho)-\varrho N}\chi_{j}(\xi)$. Then we can write
\begin{eqnarray*}
K_{j}(x,y)
=\frac{1}{|x-y|^{N}}\int_{{\mathbb{R}^n}}e^{ i\langle x-y, \xi\rangle}\partial^{N}_{\xi}a_j(x,\xi)d\xi
=\frac{1}{|x-y|^{N}}T_{\tilde{c}_{j}}\tilde{h}_{j}(x-y)
\end{eqnarray*}
So, the second term in (\ref{E01}) can be wrote as
\begin{eqnarray*}
\int_{|y-x_{0}|>2^{-j\varrho+1}}\big(\int_{|x-x_{0}|<l}
|\frac{1}{|y-x|^{N}}T_{\tilde{c}_{j}}\tilde{h}_{j}(x-y)|^{2}dx\big)^{\frac{1}{2}}|f(y)|dy.
\end{eqnarray*}
Notice that $|y-x|\sim|y-x_{0}|$ for any $|x-x_{0}|<l$ and $|y-x_{0}|>2^{-j\varrho+1}\geq2l$. Then it is bounded by
\begin{eqnarray*}
\int_{|y-x_{0}|>2^{-j\varrho+1}}\frac{1}{|y-x_{0}|^{N}}\big(\int_{\mathbb{R}^{n}}
|T_{\tilde{c}_{j}}\tilde{h}_{j}(x-y)|^{2}dx\big)^{\frac{1}{2}}|f(y)|dy.
\end{eqnarray*}
Clearly, $\tilde{c_{j}}\in S^{-\frac{n}{2}\max\{\delta-\varrho,0\}}_{\varrho,\delta}$. So $L^{2}$ boundedness of $T_{\tilde{c_{j}}}$ gives that it has bound
\begin{eqnarray*}
\int_{|y-x_{0}|>2^{-j\varrho+1}}\frac{1}{|y-x_{0}|^{N}}|f(y)|dy\big(\int_{\mathbb{R}^{n}}|\tilde{h}_{j}(\xi)|^{2}d\xi\big)^{\frac{1}{2}}
\leq 2^{-j\frac{n}{2}}Mf(x_{0}).
\end{eqnarray*}
For (\ref{E7}), it be got by the same argument as above with $L^{2}$ boundedness of pseudo-differential operators replaced by Parseval's identity. So the proof is finished.
\end{proof}
We remark that there is no use for the smoothness of variable $y$ of $a(y,\xi)$ when we prove (\ref{E7}). So the condition on $a(y,\xi)$ can be relaxed. More exactly, we have
\begin{lem}\label{La6}
Let $Q(x_{0},l)$ be a fixed cube with side length $l<1$. Suppose $0<\varrho\leq1$. For any positive integer $j$ with $l^{-1}\leq2^{j}\leq l^{-\frac{1}{\varrho}}$, if $a\in L^{\infty}S^{-n(1-\varrho)}_{\varrho}$ then
\begin{eqnarray}\label{E18}
\frac{1}{|Q(x_{0},l)|}\int_{Q(x_{0},l)}|T^{*}_{j}f(x)|dx
&\lesssim&2^{-j\frac{n}{2}}l^{-\frac{n}{2}}Mf(x_{0}).
\end{eqnarray}
\end{lem}

\begin{lem}\label{La5}
Let $Q(x_{0},l)$ be a fixed cube with side length $l<1$. Suppose $0<\varrho<\delta<1$. For for any $1\leq\lambda\leq\frac{1}{\varrho}$ and any positive integer $j$ with $l^{-\lambda}\leq2^{j}\leq l^{-\frac{1}{\varrho}}$, if $a\in S^{-n(1-\varrho)}_{\varrho,\delta}$ then
\begin{eqnarray*}
\frac{1}{|Q(x_{0},l)|}\int_{Q(x_{0},l)}|T_{j}f(x)|dx
&\lesssim&\big(l^{\lambda}2^{j\delta}+l^{-\frac{n\lambda}{2}}2^{-j\frac{n}{2}}\big)Mf(x_{0}).
\end{eqnarray*}
\end{lem}

\begin{proof}
This proof can be completed by a similar argument as in the proof of Lemma \ref{L7}. Using the notations in them, one write
\begin{eqnarray*}
&&\frac{1}{|Q|}\int_{Q(x_{0},l)}|T_{j}f(x)|dx\nonumber\\
&\leq&\frac{1}{|Q|}\sum\limits_{i=1}^{L^{n}}\bigg(\int_{Q(x_{i},l^{\lambda})}|T_{j}f(x)-T_{j,i}f(x)|dx+\int_{Q(x_{i},l^{\lambda})}|T_{j,i}f(x)|dx\bigg).
\end{eqnarray*}
It is easy to get
\begin{eqnarray*}
|T_{j}f( x)-T_{j,i}f( x)|\lesssim l^{\lambda}2^{j\delta}Mf(x_{0})
\end{eqnarray*}
and
\begin{eqnarray*}
\int_{Q(x_{i},l^{\lambda})}|T_{j,i}f(x)|dx
&\lesssim&2^{-j\frac{n}{2}}l^{\frac{n\lambda}{2}}Mf(x_{0}).
\end{eqnarray*}
Recall $L^{n}\leq2^{n}l^{n(1-\lambda)}$, the desired estimate can be gotten immediately.
\end{proof}
%

Applying weak $(1,1)$ estimate for $T_{j}$ and Kolmogorov's inequality instead of $L^{p}$ estimate in the proof Lemma \ref{L70}, we can get a similar result for $\varrho>0$.
\begin{lem}\label{L1}
Let $Q(x_{0},l)$ be a fixed cube with side length $l<1$. Suppose  $0<\varrho\leq1$, $0\leq\delta<1$ and $0<\epsilon<1$. For any positive integer $N>n$ and any positive integer $j$ with $l^{-\frac{1}{\varrho}}\leq2^{j}$,
if $a\in S^{-n(1-\varrho)}_{\varrho,\delta}$ then
\begin{eqnarray}\label{E8}
\frac{1}{|Q|}\int_{Q(x_{0},l)}|T_{j}u(x)|^{\epsilon}dx
\lesssim\big(2^{-j(\frac{n}{2}(1-\varrho)-\frac{n}{2}\max\{\delta-\varrho,0\})\epsilon}+2^{-j\varrho(n-N)\epsilon}l^{n-N)\epsilon}\big)
\big(Mu(x_{0})\big)^{\epsilon}
\end{eqnarray}
and
\begin{eqnarray}\label{E9}
\frac{1}{|Q|}\int_{Q(x_{0},l)}|T^{*}_{j}u(x)|^{\epsilon}dx
\lesssim\big(2^{-j(\frac{n}{2}(1-\varrho)-\frac{n}{2}\max\{\delta-\varrho,0\})\epsilon}+2^{-j\varrho(n-N)\epsilon}l^{n-N)\epsilon}\big)
\big(Mu(x_{0})\big)^{\epsilon}.
\end{eqnarray}
\end{lem}

\begin{proof}[Proof of Theorem \ref{TH3}]

Without loss of generality, we assume that the symbol $a(x,\xi)$ vanishes for $|\xi|\leq 1$. Let $Q=Q(x_{0},l)$ denote the cube centered at $x_{0}$ with the side length $l.$ For any fixed cube $Q$, we are going to prove that
\begin{eqnarray*}
\frac{1}{|Q|}\int_{Q}||T_{a}u(x)|^{\epsilon}-|C_{Q}|^{\epsilon}|dx\lesssim \big(Mu(x_{0})\big)^{\epsilon},
\end{eqnarray*}
where $C_{Q}=\frac{1}{|Q|}\int_{Q}T_{a}u(y)dy$. Notice that $||a|^{\epsilon}-|b|^{\epsilon}|\leq|a-b|^{\epsilon}$, $0<\epsilon<1$, it suffices to prove
\begin{eqnarray}\label{E00}
\frac{1}{|Q|}\int_{Q}|T_{a}u(x)-C_{Q}|^{\epsilon}dx\lesssim \big(Mu(x_{0})\big)^{\epsilon}.
\end{eqnarray}
Clearly, the left hand integral in (\ref{E00}) for any $0<\epsilon<1$ can be bounded by
\begin{eqnarray*}
\sum_{j}\frac{2}{|Q|^{2}}\int_{Q}\int_{Q}|T_{j}u(x)-T_{j}u(y)|^{\epsilon}dydx.
\end{eqnarray*}
Then by the same argument as the proof Theorem \ref{TH3}, we can get the desired estimate.
\end{proof}

\begin{proof}[Proof of Theorem \ref{TH5}]
We give a outline here, since it can be proved by a similar argument as above. Clearly, it suffices to show
\begin{eqnarray*}
\sum_{j}\frac{2}{|Q|^{2}}\int_{Q}\int_{Q}|T^{*}_{j}u(x)-T^{*}_{j}u(y)|dydx\lesssim Mu(x_{0}).
\end{eqnarray*}
For the case $l<1$. Break up this sum as follows
$$\sum\limits_{2^{j}<l^{-1}}+\sum\limits_{l^{-1}\leq2^{j}\leq l^{-\frac{1}{\varrho}}}
+\sum\limits_{l^{-\frac{1}{\varrho}}<2^{j}}.$$
Then, we can get the desired estimate for the first term($2^{j}<l^{-1}$) and the second term($l^{-1}\leq2^{j}\leq l^{-\frac{1}{\varrho}}$) by Remark \ref{R41} and Lemma \ref{La6} respectively. As for the last term($l^{-\frac{1}{\varrho}}<2^{j}$) and the case $l>1$, it can be estimated by following lemma. So the proof is finished.
\end{proof}

\begin{lem}
Suppose  $0<\varrho<1$. For any positive integer $N>n$,
if $a\in L^{\infty}S^{-n(1-\varrho)}_{\varrho}$ then for $0<\theta<\frac{n}{2}(1-\varrho)$
\begin{eqnarray}
\frac{1}{|Q|}\int_{Q(x_{0},l)}|T^{*}_{j}u(x)|dx
\lesssim\big(2^{-j(\frac{n}{2}(1-\varrho)-\theta)}+2^{-j\varrho(n-N)}l^{n-N)}\big)
Mu(x_{0}).
\end{eqnarray}
\end{lem}
\begin{proof}
We show a outline here.
Set $u_{5}(x)=u(x)\chi_{Q(x_{0},2l)}(x)$ and $u_{6}(x)=u(x)-u_{5}(x)$.
Then
\begin{eqnarray}\label{E19}
\frac{1}{|Q|}\int_{Q(x_{0},l)}|T^{*}_{j}u(x)|dx
\leq\frac{1}{|Q|}\int_{Q(x_{0},l)}|T^{*}_{j}u_{5}(x)|dx+
\frac{1}{|Q|}\int_{Q(x_{0},l)}|T^{*}_{j}u_{6}(x)|dx.
\end{eqnarray}

Notice that $a_{j}(y,\xi)\in S^{-\frac{n}{2}(1-\varrho)-\theta}_{\varrho,\delta}$ with bounds $\lesssim 2^{-j\frac{n}{2}(1-\varrho)+j\theta}$. The $L^{1}$-estimate of $T^{*}_{j}$ give that
\begin{eqnarray}\label{E20}
\frac{1}{|Q|}\int_{Q}|T^{*}_{j}u_{5}(x)|dx
&\lesssim& 2^{-j\frac{n}{2}(1-\varrho)+j\theta}Mu(x_{0}).
\end{eqnarray}

Notice that $|y-x|\sim|y-x_{0}|$ for $\forall x\in Q(x_{0},l)$ and $\forall y\in Q^{C}(x_{0},2l)$. So integrating by parts gives that
\begin{eqnarray}\label{E21}
|T^{*}_{j}u_{6}(x)|
\lesssim2^{-j\varrho(n-N)}l^{n-N)}
Mu(x_{0}).
\end{eqnarray}
Clearly, the desired estimate follows from \eqref{E19}, \eqref{E20} and \eqref{E21}.
\end{proof}

\section{The proof of continuity on Hardy spaces}
A tempered distribution $f$ belongs to Hardy spaces $H^{p}(\mathbb{R}^{n})$ if, for some $\phi\in\mathscr{S}$ with $\int_{\mathbb{R}^{n}}\phi(x)dx\neq0$, the maximal operator
$$\mathscr{M}_{\phi}f(x):=\sup\limits_{t>0}|f*\phi_{t}(x)|$$
is in $L^{p}(\mathbb{R}^{n})$, where $\phi_{t}(x)=t^{-n}\phi(x/t)$. The continuity properties of pseudo-differential operator $T_{a}$ and operators $T^{*}_{a}$ acting on Hardy spaces $H^{p}(\mathbb{R}^{n})$ will be done by standard atomic and molecular technique \cite{Taibleson}.
\begin{defn}
Let $0<p\leq1\leq q\leq\infty$, $p\neq q$, and the nonnegtive integer $s\geq[n(\frac{1}{p}-1)]$. A function $a(x)\in L^{q}(\mathbb{R}^{n})$ is called a $(p,q,s)$ atom with the center at $x_{0}$, if it satisfies the following conditions:
\begin{equation*}
\begin{array}{c}
  \displaystyle (1)~\supp a_{Q}\subset Q; \quad (2)~\int_{\mathbb{R}^{n}}|a_{Q}(y)|^{q}\leq |Q|^{1-\frac{q}{p}};\quad (3)~\int_{\mathbb{R}^{n}}a_{Q}(y)y^{\alpha}dy=0, 0\leq|\alpha|\leq s.
\end{array}
\end{equation*}
\end{defn}

\begin{defn}(Taibleson and Weiss \cite{Taibleson})
Let $0<p\leq1\leq q\leq\infty$, $p\neq q$, and the nonnegtive integer $s\geq[n(\frac{1}{p}-1)]$, $\epsilon>\max\{\frac{s}{n},\frac{1}{p}-1\}$,$a_{0}=1-\frac{1}{p}+\epsilon$ and $b_{0}=1-\frac{1}{q}+\epsilon$. A $(p,q,s,\epsilon)$ molecule center at $x_{0}$ is a function $M$ such that $M(x)\in L^{q}(\mathbb{R}^{n})$ and $|x|^{nb_{0}}M(x)\in L^{q}(\mathbb{R}^{n})$ satisfying:
\begin{equation*}
\begin{array}{c}
  \displaystyle (1)~\|M\|^{a_{0}}_{L^{q}}\|M(\cdot)|\cdot-x_{0}|^{nb_{0}}\|^{b_{0}-a_{0}}_{L^{q}}<\infty;\quad
(2)\int_{\mathbb{R}^{n}}M(x)x^{\alpha}dx=0, 0\leq|\alpha|\leq s.
\end{array}
\end{equation*}
\end{defn}

To prove Theorem \ref{TH6}, it suffices to show
\begin{prop}\label{P2}
Let $a_{Q}$ be a $(p,2,2t)$ atom with $0<p<1$ and $t$ be an even integer $t>\frac{n}{p}$.
\begin{enumerate}
  \item If $T^{*}_{a}$ defined as (\ref{df*}) satisfies condition (1) in Theorem \ref{TH6}. Then $T^{*}_{a}a_{Q}$ is a $(p,1,[n(\frac{1}{p}-1)],\frac{t}{n}-\frac{1}{2})$ molecule.
  \item If $T_{a}$ defined as (\ref{df}) satisfies condition (2) in Theorem \ref{TH6}. Then $T_{a}a_{Q}$ is a $(p,1,[n(\frac{1}{p}-1)],\frac{t}{n}-\frac{1}{2})$ molecule.
\end{enumerate}
\end{prop}

\begin{lem}\label{La2}
Let $0<p\leq1$, $t\geq [n(\frac{1}{p}-1)]$ and $a_{Q}$ is a $(p,2,2t)$-atom with the center at the origin and $Q=Q(0,l)$ is a cube on which $a_{Q}$ is supported. Suppose $0<l<1$, $0\leq\varrho\leq1$ and $0\leq\delta<1$. For any positive integer $j$ with $2^{j}\leq l^{-1}$ and any positive integer $2N_{1}>\frac{n}{2}$, if $a\in S^{-n(1-\varrho)(\frac{1}{p}-\frac{1}{2})}_{\varrho,\delta}$ then
\begin{eqnarray}
\int_{\mathbb{R}^{n}}|T^{*}_{j}a_{Q}(x)|^{q}dx
&\lesssim&2^{jqn\big(\frac{t}{2N_{1}}(\frac{1}{q}-\frac{1}{2})+\varrho(\frac{1}{p}-\frac{1}{q})+(1-\frac{1}{p})\big)}
l^{qn\big(\frac{t}{2N_{1}}(\frac{1}{q}-\frac{1}{2})+(1-\frac{1}{p})\big)};\label{E14}\\
\!\!\!\!\!\!\!\!\!\int_{\mathbb{R}^{n}}|x|^{qt}|T^{*}_{j}a_{Q}(x)|^{q}dx
\!\!&\lesssim&\!\!2^{jqn\big(\frac{t}{2N_{1}}(\frac{1}{q}-\frac{1}{2})+\varrho(\frac{1}{p}-\frac{1}{q})+(1-\frac{1}{p})\big)+jqt(1-\varrho)}
l^{qn\big(\frac{t}{2N_{1}}(\frac{1}{q}-\frac{1}{2})+(1-\frac{1}{p})\big)+qt}.\label{e140}
\end{eqnarray}
\end{lem}

\begin{proof}
We prove (\ref{E14}) first. Denote
$$T=l^{\frac{t}{2N_{1}}}2^{j\frac{t}{2N_{1}}-j\varrho},$$
and break up the integral with respect to the variable $x$ as follows
\begin{eqnarray}\label{d11}
\int_{|x|\leq 2T}
+\int_{|x|> 2T}.
\end{eqnarray}
Next, we show that both of them is bounded by the right hand in \ref{E14}.
 H\"{o}lder's inequality and Minkowski's inequality show that the first integral in (\ref{d11})is bounded by
\begin{eqnarray*}
T^{n(1-\frac{q}{2})}
\bigg(\int_{Q(0,l)}|a_{Q}(y)|\big(\int_{\mathbb{R}^{n}}|\int_{\mathbb{R}^n} e^{ i \langle x-y,\xi\rangle}a(y,\xi) d\xi|^{2} dx\big)^{\frac{1}{2}}dy\bigg)^{q}.
\end{eqnarray*}
Recall $a_{Q}$ is a $(p,2,2t)$-atom and $T=l^{\frac{t}{2N_{1}}}2^{j\frac{t}{2N_{1}}-j\varrho}$. Then the desired estimate can be get by Parseval's identity.

Next, we estimate the second integral in (\ref{d11}). Integrating by parts gives that for any multi-index $\alpha$ with $|\alpha|=N_{1}$
\begin{eqnarray*}
&&\int_{\mathbb{R}^n}e^{ i \langle x,\xi\rangle}\int_{\mathbb{R}^n} e^{ -i \langle y,\xi\rangle}a_{j}(y,\xi) a_{Q}(y)dy d\xi\\
&=&|x|^{-2N_{1}}\sum\limits_{|\alpha_{1}|+|\alpha_{2}|=|\alpha|}\int_{\mathbb{R}^n}e^{ i \langle x,\xi\rangle}\int_{\mathbb{R}^n}e^{ -i \langle y,\xi\rangle}(\triangle_{\xi})^{\alpha_{2}}\big(a_{j}(y,\xi) \big) y^{2\alpha_{1}}a_{Q}(y)dy d\xi.
\end{eqnarray*}
For any fixed $\xi\in \mathbb{R}^{n}$, let $P_{\xi}$ be the Taylor polynomial in $y$ of degree $t-2|\alpha_{1}|-1$ of $e^{ -i \langle y,\xi\rangle}(\triangle_{\xi})^{\alpha_{2}}\big(a_{j}(y,\xi) \big)$ about the origin. Then
\begin{eqnarray*}
&&\int_{\mathbb{R}^n}e^{ -i \langle y,\xi\rangle}(\triangle_{\xi})^{\alpha_{2}}\big(a_{j}(y,\xi) \big) y^{2\alpha_{1}}a_{Q}(y)dy \\
&=& \int_{\mathbb{R}^n}\big(e^{ -i \langle y,\xi\rangle}(\triangle_{\xi})^{\alpha_{2}}a_{j}(y,\xi)-P(y)\big)y^{2\alpha_{1}}a_{Q}(y)dy\\
&=& \sum\limits_{|\beta_{1}|+|\beta_{2}|=|\beta|}\int_{\mathbb{R}^n}e^{ -i \langle \bar{y},\xi\rangle}\xi^{\beta_{1}}\big(\partial^{\beta_{2}}_{y}(\triangle_{\xi})^{\alpha_{2}} \big)\big(a_{j}(\bar{y},\xi)\big)y^{2\alpha_{1}+\beta}a_{Q}(y)dy,
\end{eqnarray*}
where $|\beta|=t-2|\alpha_{1}|$ and $\bar{y}$ is a point around the origin. Therefore, we can write
\begin{eqnarray}\label{f1}
T^{*}_{j}a_{Q}(x)\nonumber
&=&\frac{1}{|x|^{2N_{1}}}\sum\limits_{|\alpha_{1}|+|\alpha_{2}|=|\alpha|}\sum\limits_{|\beta_{1}|+|\beta_{2}|=|\beta|}\\
&\times&\int_{\mathbb{R}^n}\int_{\mathbb{R}^n}e^{ i \langle x-\bar{y},\xi\rangle}\xi^{\beta_{1}}\big(\partial^{\beta_{2}}_{y}(\triangle_{\xi})^{\alpha_{2}} \big)\big(a_{j}(\bar{y},\xi)\big)y^{2\alpha_{1}+\beta}a_{Q}(y)d\xi dy.
\end{eqnarray}
By  H\"{o}lder's inequality and Minkowski's inequality, the second integral in (\ref{d11}) is bounded by
\begin{eqnarray*}
&&\sum\limits_{|\alpha_{1}|+|\alpha_{2}|=|\alpha|}\sum\limits_{|\beta_{1}|+|\beta_{2}|=|\beta|}
\big(\int_{|x|>T}\frac{1}{|x|^{2N_{1}(\frac{2q}{2-q})}}dx\big)^{1-\frac{q}{2}}\\
&\times&\bigg(\int_{Q(0,l)}|a_{Q}(y)|\big(\int_{\mathbb{R}^{n}}|\int_{\mathbb{R}^n} e^{ i \langle x-\bar{y},\xi\rangle}\xi^{\beta_{1}}\big(\partial^{\beta_{2}}_{y}(\triangle_{\xi})^{\alpha_{2}} \big)\big(a_{j}(\bar{y},\xi)\big)y^{2\alpha_{1}+\beta} d\xi |^{2}dx \big)^{\frac{1}{2}}dy\bigg)^{q}.
\end{eqnarray*}
Recall $a_{Q}$ is a $(p,2,2t)$-atom and $T=l^{\frac{t}{2N_{1}}}2^{j\frac{t}{2N_{1}}-j\varrho}$. Then the desired estimate can be get by Parseval's identity.

The main idea to prove (\ref{e140}) is writing $|x|^{t}|T^{*}_{j}a_{Q}(x)|$ as sum of $T^{*}$ first, then following the same method as above to estimate these operator. To this end, we fixed $|\alpha|=\frac{t}{2}$ in (\ref{f1}) and give a clear relationship between $y$ and $\bar{y}$.
We write
\begin{eqnarray*}
&&e^{ -i \langle y,\xi\rangle}(\triangle_{\xi})^{\alpha_{2}}\big(a_{j}(y,\xi) \big)\\
&=&P_{\xi}(y)+C_{t}\int^{1}_{0}(1-\theta)^{t-2|\alpha_{2}|}\bigg(\partial^{2\alpha-2\alpha_{2}}_{y}
e^{ -i \langle \cdot,\xi\rangle}(\triangle_{\xi})^{\alpha_{2}}\big(a_{j}(\cdot,\xi) \big)\bigg)(\theta y)y^{2\alpha-2\alpha_{2}}d\theta
\end{eqnarray*}
The cancellation condition of $a_{Q}$ gives
\begin{eqnarray*}
|x|^{t}T^{*}_{j}a_{Q}(x)&=&C_{t}
\sum\limits_{|\alpha_{1}|+|\alpha_{2}|=|\alpha|}
\int^{1}_{0}(1-\theta)^{t-2|\alpha_{2}|}\sum\limits_{|\beta_{1}|+|\beta_{2}|=|\beta|}\\
&\times&\int_{\mathbb{R}^n}\int_{\mathbb{R}^n}e^{ i \langle x-\theta y,\xi\rangle}\xi^{\beta_{1}}\big(\partial^{\beta_{2}}_{y}(\triangle_{\xi})^{\alpha_{2}} \big)\big(a_{j}(\theta y,\xi)\big)y^{2\alpha_{1}+\beta}a_{Q}(y)d\xi dyd\theta,
\end{eqnarray*}
where $|\beta|=t-2|\alpha_{1}|$. Denote
$a_{j,\beta_{1},\beta_{2}}(y,\xi)=\xi^{\beta_{1}}\big(\partial^{\beta_{2}}_{y}(\triangle_{\xi})^{\alpha_{2}} \big)\big(a_{j}(y,\xi)\big)$
and $f_{\alpha_{1},\beta,Q}(y)=y^{2\alpha_{1}+\beta}a_{Q}(y)$.
Then
\begin{eqnarray*}
|x|^{t}T^{*}_{j}a_{Q}(x)&=&C_{t}
\sum\limits_{|\alpha_{1}|+|\alpha_{2}|=|\alpha|}
\int^{1}_{0}(1-\theta)^{t-2|\alpha_{2}|}\theta^{-n}\sum\limits_{|\beta_{1}|+|\beta_{2}|=|\beta|}
\big(T^{*}_{a_{j,\beta_{1},\beta_{2}}}f_{\alpha_{1},\beta,Q}(\theta\cdot)\big)(x)d\theta,
\end{eqnarray*}
Moreover,
\begin{eqnarray}\label{E22}
\int_{\mathbb{R}^{n}}|x|^{qt}|T^{*}_{j}a_{Q}(x)|^{q}dx
&\lesssim&\bigg(\sum\limits_{|\alpha_{1}|+|\alpha_{2}|=|\alpha|}
\int^{1}_{0}(1-\theta)^{t-2|\alpha_{2}|}\theta^{-n}\nonumber\\
&\times&\sum\limits_{|\beta_{1}|+|\beta_{2}|=|\beta|}
\big(\int_{\mathbb{R}^{n}}|\big(T^{*}_{a_{j,\beta_{1},\beta_{2}}}f_{\alpha_{1},\beta,Q}(\theta\cdot)\big)(x)|^{q}dx\big)^{\frac{1}{q}}d\theta\bigg)^{q}
,
\end{eqnarray}
Notice that
$a_{j,\beta_{1},\beta_{2}}\in  S^{-n(1-\varrho)(\frac{1}{p}-\frac{1}{2})+t(1-\varrho)}_{\varrho,\delta}$ with its normal independent of $j,\beta_{1},\beta_{2}$
and $f_{\alpha_{1},\beta,Q}(x)$ satisfies $\supp f_{\alpha_{1},\beta,Q}\subset Q$, $\int_{\mathbb{R}^{n}}|f_{\alpha_{1},\beta,Q}(x)|dx\leq |Q|^{1-\frac{1}{p}+\frac{t}{n}}$ and $\int_{\mathbb{R}^{n}}f_{\alpha_{1},\beta,Q}(y)y^{\alpha}dy=0, 0\leq|\alpha|\leq t$, since $a_{Q}$ is a $(p,2,2t)$ atom.

By the same argument as (\ref{E14}), we can get
\begin{eqnarray*}
\int_{\mathbb{R}^{n}}|\big(T^{*}_{a_{j,\beta_{1},\beta_{2}}}f_{\alpha_{1},\beta,Q}(\theta\cdot)\big)(x)|^{q}dx
&\lesssim& \theta^{qn}2^{jqn\big(\frac{t}{2N_{1}}(\frac{1}{q}-\frac{1}{2})+\varrho(\frac{1}{p}-\frac{1}{q})+(1-\frac{1}{p})\big)+jqt(1-\varrho)}\\
&\times&l^{qn\big(\frac{t}{2N_{1}}(\frac{1}{q}-\frac{1}{2})+(1-\frac{1}{p})\big)+qt}.
\end{eqnarray*}
By substituting this into (\ref{E22}), the desired estimate can be gotten immediately, since that $\int^{1}_{0}(1-\theta)^{t-2|\alpha_{2}|}d\theta\lesssim1$ is always true for $2|\alpha_{1}|\leq t$.
\end{proof}

\begin{lem}\label{La1}
Let $Q(0,l)$ be a fixed cube with side length $l<1$. Suppose $0<q<2$, $0<p<1$, $0<\varrho<1$ and $0\leq\delta<1$. For any positive integer $j$ with $l^{-1}\leq2^{j}<l^{-\frac{1}{\varrho}}$ if $a\in S^{-n(1-\varrho)(\frac{1}{p}-\frac{1}{2})}_{\varrho,\delta}$ then
\begin{eqnarray}
\int_{\mathbb{R}^{n}}|T^{*}_{j}a_{Q}(x)|^{q}dx
&\lesssim&2^{jqn\big(\varrho(\frac{1}{p}-\frac{1}{q})+(1-\frac{1}{p})\big)}l^{qn(1-\frac{1}{p})};\label{E16}\\
\int_{\mathbb{R}^{n}}|x|^{qt}|T^{*}_{j}a_{Q}(x)|^{q}dx
&\lesssim&2^{jqn\big(\varrho(\frac{1}{p}-\frac{1}{q})+(1-\frac{1}{p})\big)-jqt\varrho}l^{qn(1-\frac{1}{p})}.\label{E23}
\end{eqnarray}
\end{lem}

\begin{proof}
We proof (\ref{E16}) first. Break up the integral with respect to the variable $x$ as follows
\begin{eqnarray}\label{d0}
\int_{|x|\leq 2^{-j\varrho+1}}
+\int_{|x|> 2^{-j\varrho+1}}.
\end{eqnarray}
 H\"{o}lder's inequality and Parseval's identity show that the first integral in (\ref{d0})is bounded by
\begin{eqnarray}\label{E27}
&&\big(\int_{|x|\leq 2^{-j\varrho}}dx\big)^{1-\frac{q}{2}}
\big(\int_{\mathbb{R}^{n}}|\int_{Q(0,l)}\int_{\mathbb{R}^n} e^{ i \langle x-y,\xi\rangle} a_{j}(y,\xi)a_{Q}(y) d\xi dy|^{2}dx\big)^{\frac{q}{2}}\\
&\leq&2^{-jn\varrho(1-\frac{q}{2})}
\bigg(\int_{Q(0,l)}|a_{Q}(y)|\big(\int_{\mathbb{R}^{n}}|\int_{\mathbb{R}^n} e^{ i \langle x-y,\xi\rangle}a_{j}(y,\xi) d\xi|^{2} dx\big)^{\frac{1}{2}}dy\bigg)^{q}\nonumber\\
&\lesssim&2^{-jn\varrho(1-\frac{q}{2})}
\bigg(\int_{Q(0,l)}|a_{Q}(y)|\big(\int_{\mathbb{R}^n} | a_{j}(y,\xi)|^{2} d\xi\big)^{\frac{1}{2}}dy\bigg)^{q}\nonumber\\
&\lesssim&2^{-jn\varrho(1-\frac{q}{2})+jq(-n(1-\varrho)(\frac{1}{p}-\frac{1}{2})+\frac{n}{2})}l^{q(n-\frac{n}{p})}.\nonumber
\end{eqnarray}

By H\"{o}lder's inequality, integrating by parts, Parseval's identity and the fact $|x|\sim|x-y|$ that follows from $l<1$, $y\in Q(0,l)$ and $|x|>2^{-j\varrho+1}$. The second integral in (\ref{d0}) is bounded by
\begin{eqnarray}\label{E28}
&\big(&\int_{|x|>2^{-j\varrho+1}}\frac{1}{|x|^{N(\frac{2q}{2-q})}}dx\big)^{1-\frac{q}{2}}\nonumber\\
&\times&\big(\int_{|x|>2^{-j\varrho+1}}|x|^{2N}|\int_{Q(0,l)}\int_{\mathbb{R}^n} e^{ i \langle x-y,\xi\rangle} a_{j}(y,\xi)a_{Q}(y) d\xi dy|^{2}dx\big)^{\frac{q}{2}}\\
&\lesssim&2^{-j\varrho q\big(n(\frac{1}{q}-\frac{1}{2})-N\big)}
\bigg(\int_{Q(0,l)}|a_{Q}(y)|\big(\int_{\mathbb{R}^{n}}|x-y|^{2N}|\int_{\mathbb{R}^n} e^{ i \langle x-y,\xi\rangle}a_{j}(y,\xi) d\xi|^{2} dx\big)^{\frac{1}{2}}dy\bigg)^{q}\nonumber\\
&\lesssim&2^{-j\varrho q\big(n(\frac{1}{q}-\frac{1}{2})-N\big)}
\bigg(\int_{Q(0,l)}|a_{Q}(y)|\big(\int_{\mathbb{R}^n} |\partial^{\alpha}_{\xi}a_{j}(y,\xi)|^{2} d\xi\big)^{\frac{1}{2}}dy\bigg)^{q}\nonumber\\
&\lesssim&2^{-jq\varrho\big(n(\frac{1}{q}-\frac{1}{2})-N\big)+jq(-n(1-\varrho)(\frac{1}{p}-\frac{1}{2})+\frac{n}{2}-\varrho N)}l^{q(n-\frac{n}{p})}.\nonumber
\end{eqnarray}

The proof of (\ref{E23}) is a little different from (\ref{E16}), that is, the first term in (\ref{E27}) and (\ref{E28}) is
$$\big(\int_{|x|\leq 2^{-j\varrho}}|x|^{t\frac{2q}{2-q}}dx\big)^{1-\frac{q}{2}}\quad \mathrm{and} \quad
\big(\int_{|x|>2^{-j\varrho+1}}\frac{1}{|x|^{(N-t)(\frac{2q}{2-q})}}dx\big)^{1-\frac{q}{2}}.$$
\end{proof}
In the course of the above proof, if $\varrho=0$, $|x|\sim|x-y|$ is still true for $l<1$, $y\in Q(0,l)$ and $|x|>2$. Thus we have
\begin{lem}\label{R3}
Let $Q(0,l)$ be a fixed cube with side length $l<1$. Suppose $0<q<2$, $0<p<1$ and $0\leq\delta<1$. For any positive integer $j$ with $l^{-1}\leq2^{j}$ if $a\in S^{-n(\frac{1}{p}-\frac{1}{2})}_{0,\delta}$ then
\begin{eqnarray}
\int_{\mathbb{R}^{n}}|T^{*}_{j}a_{Q}(x)|^{q}dx
&\lesssim&2^{jqn(1-\frac{1}{p})}l^{qn(1-\frac{1}{p})};\label{E1600}\\
\int_{\mathbb{R}^{n}}|x|^{qt}|T^{*}_{j}a_{Q}(x)|^{q}dx
&\lesssim&2^{jqn(1-\frac{1}{p})}l^{qn(1-\frac{1}{p})}.\label{E2300}
\end{eqnarray}
\end{lem}

\begin{lem}\label{L10}
Suppose $0<q<2$, $0<p\leq1$, $0\leq\varrho<1$ and $0\leq\delta<1$. For any positive integer $2N_{2}>\frac{n(2-p)}{2}$, if $a\in S^{-n(1-\varrho)(\frac{1}{p}-\frac{1}{2})}_{\varrho,\delta}$ then

\begin{eqnarray}
\int_{\mathbb{R}^{n}}|T^{*}_{j}a_{Q}(x)|^{q}dx
&\lesssim&
l^{qn(\frac{1}{q}-\frac{1}{p})}2^{-jq\big(n(1-\varrho)(\frac{1}{p}-\frac{1}{2})-\frac{n}{2}\max(0,\delta-\varrho)\big)}\nonumber\\
&+&l^{q\big(n(\frac{1}{q}-\frac{1}{p}+\frac{1}{2})-N_{2}\big)}2^{jq(-n(1-\varrho)(\frac{1}{p}-\frac{1}{2})+\frac{n}{2}-\varrho N_{2})}\label{E29}\\
\int_{\mathbb{R}^{n}}|x|^{qt}|T^{*}_{j}a_{Q}(x)|^{q}dx
&\lesssim&l^{qn(\frac{1}{q}-\frac{1}{p})+qt}2^{-jq\big(n(1-\varrho)(\frac{1}{p}-\frac{1}{2})-\frac{n}{2}\max(0,\delta-\varrho)\big)}\nonumber\\
&+&l^{q\big(n(\frac{1}{q}-\frac{1}{p}+\frac{1}{2})-N_{2}\big)}2^{jq\big(-n(1-\varrho)(\frac{1}{p}-\frac{1}{2})+\frac{n}{2}-\varrho (N_{2}+t)\big)}.\label{E31}
\end{eqnarray}
\end{lem}

\begin{proof}
Break up the integral with respect to the variable $x$ as follows
\begin{eqnarray}\label{d2}
\int_{|x|\leq 2l}
+\int_{|x|> 2l}.
\end{eqnarray}
Notice that $a_{j}(y,\xi) \in S^{-\frac{n}{2}\max(0,\delta-\varrho)}_{\varrho,\delta}$ with bounds $\lesssim 2^{-jn(1-\varrho)(\frac{1}{p}-\frac{1}{2})+j\frac{n}{2}\max(0,\delta-\varrho)}$. H\"{o}lder's inequality and the $L^{2}$-estimate of $T_{j}$ give that
 the first integral in (\ref{d2})is bounded by
\begin{eqnarray}
&&\big(\int_{|x|\leq 2l}dx\big)^{1-\frac{q}{2}}\|T^{*}_{j}a_{Q}\|^{q}_{L^{2}}\label{E32}\\
&\leq&l^{n(1-\frac{q}{2})}
2^{-jq\big(n(1-\varrho)(\frac{1}{q}-\frac{1}{2})-\frac{n}{2}\max(0,\delta-\varrho)\big)}\|a_{Q}\|^{q}_{L^{2}}\nonumber\\
&\lesssim&l^{qn(\frac{1}{q}-\frac{1}{p})}2^{-jq\big(n(1-\varrho)(\frac{1}{p}-\frac{1}{2})-\frac{n}{2}\max(0,\delta-\varrho)\big)}.\nonumber
\end{eqnarray}
By H\"{o}lder's inequality, integrating by parts, Parseval's identity and the fact  $|x|\sim|x-y|$ that follows from $l<1$, $y\in Q(0,l)$ and $|x|>2l$. The second integral in (\ref{d2}) is bounded by
\begin{eqnarray}
&\big(&\int_{|x|>2l}\frac{1}{|x|^{N_{2}(\frac{2q}{2-q})}}dx\big)^{1-\frac{q}{2}}\nonumber\\
&\times&\big(\int_{|x|>2l}|x|^{2N_{2}}|\int_{Q(0,l)}\int_{\mathbb{R}^n} e^{ i \langle x-y,\xi\rangle} a_{j}(y,\xi)a_{Q}(y) d\xi dy|^{2}dx\big)^{\frac{q}{2}}\label{E33}\\
&\lesssim&l^{q\big(n(\frac{1}{q}-\frac{1}{p}+\frac{1}{2})-N_{2}\big)}2^{jq(-n(1-\varrho)(\frac{1}{p}-\frac{1}{2})+\frac{n}{2}-\varrho N_{2})}.\nonumber
\end{eqnarray}

The proof of (\ref{E31}) is a little different from (\ref{E29}), that is,  (\ref{E32}) and (\ref{E33}) in this case is
$$\big(\int_{|x|\leq 2l}|x|^{t\frac{2p}{2-p}}dx\big)^{1-\frac{p}{2}}\|T^{*}_{j}a_{Q}\|^{p}_{L^{2}}$$
and
\begin{eqnarray*}
&\big(&\int_{|x|>2l}\frac{1}{|x|^{N_{2}(\frac{2p}{2-p})}}dx\big)^{1-\frac{p}{2}}\nonumber\\
&\times&\big(\int_{|x|>2l}|x|^{2(N_{2}+t)}|\int_{Q(0,l)}\int_{\mathbb{R}^n} e^{ i \langle x-y,\xi\rangle} a_{j}(y,\xi)a_{Q}(y) d\xi dy|^{2}dx\big)^{\frac{p}{2}},
\end{eqnarray*}
respectively.
\end{proof}

\begin{rem}\label{R4}
Lemma \ref{La1} and Lemma \ref{L10} are still valid when $\varrho=1$, but both of then can not be used. In fact, there is no $T^{*}_{j}$ in Lemma \ref{La1} and no convergence factor in Lemma \ref{L10}. However, the $H^{p}$-continuity in this paper can be proved without them.
\end{rem}
\begin{rem}\label{R5}
Lemma \ref{La2}, Lemma \ref{La1} and Lemma \ref{L10} hold for $T_{j}$. They can be proved parallelly, provided in the argument above we apply the $L^{2}$-estimate for pseudo-differential operators instead of Parseval's identity at cost of $a\in S^{-n(1-\varrho)(\frac{1}{p}-\frac{1}{2})-\frac{n}{2}\max(0,\delta-\varrho)}_{\varrho,\delta}$.
\end{rem}

\begin{proof}[Proof of Proposition \ref{P2}]
(1) is considered first. By standard molecular technique, it is suffices to show that if $a_{Q}$ be a $(p,2,2t)$ atom with $t$ an even integer $t>\frac{n}{p}$, then $T^{*}a_{Q}$ is a $(p,1,s,\epsilon)$ molecule, where $s=[n(\frac{1}{p}-1)]$. Without loss of generality, we assume $Q=Q(0,l)$. Take $\epsilon=\frac{t}{n}-\frac{1}{2}$(clearly,$\epsilon>\max\{\frac{s}{n},\frac{1}{p}-1\}$), then $a_{0}=1-\frac{1}{p}+\frac{t}{n}$ and $b_{0}=\frac{t}{n}$. The vanishing of $T^{*}a_{Q}$ is clear. So it has to be shown that
$$\|T^{*}a_{Q}\|^{1-\frac{1}{p}+\frac{t}{n}}_{L^{1}}\||\cdot|^{t}T^{*}a_{Q}(\cdot)\|^{\frac{1}{p}-1}_{L^{1}}<\infty$$
To this end, it suffices to show the following inequalities
$$\left\{
  \begin{array}{ll}
    \|T^{*}a_{Q}\|_{L^{1}}\lesssim l^{\varrho (n-\frac{n}{p})}\quad
\mathrm{and}\quad \||\cdot|^{t}T^{*}a_{Q}(\cdot)\|_{L^{1}}\lesssim l^{\varrho(t+n-\frac{n}{p})}, & \mathrm{if}\hbox{ $0<l<1$;}    \\
    \|T^{*}a_{Q}\|_{L^{1}}\lesssim l^{(n-\frac{n}{p})}
\quad\mathrm{and}\quad\||\cdot|^{t}T^{*}a_{Q}(\cdot)\|_{L^{1}}\lesssim l^{(t+n-\frac{n}{p})}, & \mathrm{if}\hbox{ $l\geq1$;}
  \end{array}
\right.
$$
We compose the operator $T_{a}$ as (\ref{de}) when $0\leq\varrho<1$, then $\|T^{*}a_{Q}\|_{L^{1}}$ and $\||\cdot|^{t}T^{*}a_{Q}(\cdot)\|_{L^{1}}$ are bounded by
$$\sum\limits_{j}\int_{\mathbb{R}^{n}}|T^{*}_{j}a_{Q}(x)|dx\quad \mathrm{and} \quad \sum\limits_{j}\int_{\mathbb{R}^{n}}|x|^{t}|T^{*}_{j}a_{Q}(x)|dx$$

Case 1. $0<l<1$;
Break up this sum as before, that is,
\begin{eqnarray}\label{Break}
\left\{
  \begin{array}{ll}
    \sum\limits_{2^{j}<l^{-1}}
+\sum\limits_{l^{-1}<2^{j}}, & \hbox{if $\varrho=0$;}    \\
\sum\limits_{2^{j}<l^{-1}}+\sum\limits_{l^{-1}\leq2^{j}\leq l^{-\frac{1}{\varrho}}}
+\sum\limits_{l^{-\frac{1}{\varrho}}<2^{j}}, & \hbox{if $0<\varrho<1$.}
  \end{array}
\right.
\end{eqnarray}

If $0<\varrho<1$. By Lemma \ref{La2}, Lemma \ref{La1} and Lemma \ref{L10} after taking $q=1$, the corresponding sum can be bounded by
\begin{eqnarray*}
&&\sum\limits_{2^{j}<l^{-1}}2^{jn\varrho(\frac{1}{p}-1)+jn(1-\frac{1}{p})+jn\frac{t}{2N}(1-\frac{p}{2})}l^{n(1-\frac{1}{p})+n\frac{t}{2N}(1-\frac{p}{2})}
+\sum\limits_{l^{-1}\leq2^{j}\leq l^{-\frac{1}{\varrho}}}2^{-jn(1-\varrho)(\frac{1}{p}-1)}l^{n(1-\frac{1}{p})}\\
&+&\sum\limits_{l^{-\frac{1}{\varrho}}<2^{j}}\bigg(l^{n(1-\frac{1}{p})}2^{-j\big(n(1-\varrho)(\frac{1}{p}-\frac{1}{2})-\frac{n}{2}\max(0,\delta-\varrho)\big)}
+l^{n(1-\frac{1}{p})+(\frac{n}{2}-N_{2})}2^{-j\big(n(1-\varrho)(\frac{1}{p}-\frac{1}{2})-\frac{n}{2}+\varrho N_{2}\big)}\bigg)
\end{eqnarray*}
Clearly, the second sum above is convergent to $l^{\varrho (n-\frac{n}{p})}$ since $n(1-\varrho)(\frac{1}{p}-1)>0$. Note that $t$ is large enough, and so we can choose suitable positive integer $2N>\frac{n}{2}$ so that $n\varrho(\frac{1}{p}-1)+n(1-\frac{1}{p})+n\frac{t}{2N}(1-\frac{p}{2})>0$ since $1-\frac{p}{2}>0$. So The first sum is convergent to $l^{\varrho (n-\frac{n}{p})}$ as well. Notice that $n(1-\varrho)(\frac{1}{p}-\frac{1}{2})-\frac{n}{2}\max(0,\delta-\varrho)>0$, the first term in last sum is convergent to
\begin{eqnarray*}
l^{n(1-\frac{1}{p})+\frac{1}{\varrho}(n(1-\varrho)(\frac{1}{p}-\frac{1}{2})-\frac{n}{2}\max(0,\delta-\varrho)}
=l^{\varrho (n-\frac{n}{p})+\frac{n}{\varrho}\big((\frac{1}{p}-\frac{1}{2})(1-\varrho)^{2}+\frac{1}{2}(-\varrho^{2}+\varrho)-\max(0,\delta-\varrho)\big)}
\end{eqnarray*}
Notice that $l<1$ and
\begin{eqnarray*}
&&\frac{n}{\varrho}\big((\frac{1}{p}-\frac{1}{2})(1-\varrho)^{2}+\frac{1}{2}(-\varrho^{2}+\varrho)-\frac{1}{2}\max(0,\delta-\varrho)\big)\\
&\geq&\frac{n}{\varrho}\big((\frac{1}{p}-\frac{1}{2})(1-\varrho)^{2}+\frac{1}{2}(-\varrho^{2}+2\varrho-1)\big)\\
&=&\frac{n}{\varrho}(\frac{1}{p}-1)(1-\varrho)^{2}>0.
\end{eqnarray*}
We can get the first term in last sum is less then $l^{\varrho (n-\frac{n}{p})}$.
Taking $N_{2}$ large enough, we get the second term in last sum is convergent to
$$l^{n(1-\frac{1}{p})+\frac{n}{2}+\frac{1}{\varrho}(n(1-\varrho)(\frac{1}{p}-\frac{1}{2})-\frac{n}{2})}
=l^{\varrho (n-\frac{n}{p})+\frac{n}{\varrho}\big((\frac{1}{p}-1)(1-\varrho)^{2}\big)}
\leq l^{\varrho (n-\frac{n}{p})}.$$

If $\varrho=0$. Lemma \ref{La2}, Lemma \ref{La1} after taking $q=1$, and Lemma \ref{R3} give that the corresponding sum can be bounded by
\begin{eqnarray*}
&&\sum\limits_{2^{j}<l^{-1}}2^{jn(1-\frac{1}{p})+jn\frac{t}{2N}(1-\frac{p}{2})}l^{n(1-\frac{1}{p})+n\frac{t}{2N}(1-\frac{p}{2})}
+\sum\limits_{l^{-1}\leq2^{j}}2^{-jn(\frac{1}{p}-1)}l^{n(1-\frac{1}{p})}\lesssim1
\end{eqnarray*}

If $\varrho=1$, we can divide $a(x,\xi)$, with respect to variate $\xi$, smoothly into two parts, that is, $a(x,\xi)=\tilde{a}_{1}(x,\xi)+\tilde{a}_{2}(x,\xi)$ with $\supp_{\xi}\tilde{a}_{1}(x,\xi)\subset\{|\xi|\leq l^{-1}\}$ and $\supp_{\xi}\tilde{a}_{2}(x,\xi)\subset\{|\xi|\geq l^{-1}\}$. For $T^{*}_{\tilde{a}_{1}}$, we compose it into $\sum_{j}T^{*}_{\tilde{a}_{1},j}$ as (\ref{de}), then $T^{*}_{\tilde{a}_{1},j}a_{Q}=0$ when $2^{j}\geq l^{-1}$ and $T^{*}_{\tilde{a}_{1},j}a_{Q}$ meat the condition of Lemma \ref{La2} when $2^{j}\leq l^{-1}$. So
\begin{eqnarray*}
\int_{\mathbb{R}^{n}}|T^{*}_{\tilde{a}_{1},j}a_{Q}(x)|dx
&\leq&\sum\limits_{2^{j}\leq l^{-1}}\int_{\mathbb{R}^{n}}|T^{*}_{\tilde{a}_{1},j}a_{Q}(x)|dx\\
&\lesssim&\sum\limits_{2^{j}<l^{-1}}2^{jn(\frac{1}{p}-1)+jn(1-\frac{1}{p})+jn\frac{t}{2N}(1-\frac{p}{2})}l^{n(1-\frac{1}{p})+n\frac{t}{2N}(1-\frac{p}{2})}
\lesssim l^{(n-\frac{n}{p})}.
\end{eqnarray*}
Notice that $\tilde{a}_{2}(y,\xi) \in S^{0}_{1,\delta}$. By the same argument as Lemma \ref{L10}, it is easy to get
\begin{eqnarray*}
\int_{\mathbb{R}^{n}}|T^{*}_{\tilde{a}_{2}}a_{Q}(x)|dx\lesssim l^{n(1-\frac{1}{p})}.
\end{eqnarray*}

Next, we show the inequality $\||\cdot|^{t}T^{*}a_{Q}(\cdot)\|_{L^{1}}\lesssim l^{\varrho(t+n-\frac{n}{p})}$. Notice that $t$ is fixed large enough and $l^{t}\leq l^{\varrho t}$ for $0\leq\varrho\leq1$, this inequality can be gotten by a similar argument as above. Here, the estimates (\ref{e140}),(\ref{E23}) and (\ref{E31}) will be applied instead of (\ref{E14}), (\ref{E16}) and (\ref{E29}).

Case 2. $l\geq1$;
(\ref{E29}) in Lemma \ref{L10} gives (after taking $q=1$ and $N_{2}$ large enough)
\begin{eqnarray*}
\sum\limits_{j}\int_{\mathbb{R}^{n}}|T^{*}_{j}a_{Q}(x)|dx
&\lesssim&\sum\limits_{j}\bigg(l^{n(1-\frac{1}{p})}2^{-j\big(n(1-\varrho)(\frac{1}{p}-\frac{1}{2})-\frac{n}{2}\max(0,\delta-\varrho)\big)}\nonumber\\
&+&l^{n(1-\frac{1}{p})+(\frac{n}{2}-N_{2})}2^{-j\big(n(1-\varrho)(\frac{1}{p}-\frac{1}{2})-\frac{n}{2}+\varrho N_{2}\big)}\bigg)
\lesssim l^{n(1-\frac{1}{p})}.
\end{eqnarray*}
(\ref{E31}) in Lemma \ref{L10} gives (after taking $q=1$ and $N_{2}$ large enough)
\begin{eqnarray*}
\sum\limits_{j}\int_{\mathbb{R}^{n}}|x|^{t}|T^{*}_{j}a_{Q}(x)|dx
&\lesssim&\sum\limits_{j}\bigg(l^{n(1-\frac{1}{p})+t}2^{-j\big(n(1-\varrho)(\frac{1}{p}-\frac{1}{2})-\frac{n}{2}\max(0,\delta-\varrho)\big)}\nonumber\\
&+&l^{n(1-\frac{1}{p})+(\frac{n}{2}-N_{2})}2^{-j\big(n(1-\varrho)(\frac{1}{p}-\frac{1}{2})-\frac{n}{2}+\varrho (N_{2}+t)\big)}\bigg)\nonumber\\
&\lesssim&l^{n(1-\frac{1}{p})+t}+l^{n(1-\frac{1}{p})}
\leq l^{n(1-\frac{1}{p})+t}.
\end{eqnarray*}

By Remark \ref{R5}, the proofs of (2) is completely parallel.
\end{proof}

\begin{proof}[Proof of Theorem \ref{TH1}]
The proofs of (1) will be shown only and the proofs of (2) is completely parallel. Here, we always assume $0\leq\varrho<1$ as the case $\varrho=1$ is considered in Theorem \ref{B3}.

The $0<p<1$ is considered first. Let nonnegtive integer $t\geq[n(\frac{1}{p}-1)]$ ($[x]$ indicates the integer part of $[x]$).
A function $a_{Q}\in L(\mathbb{R}^{n})$ is called $(p,2,t)$ atom if it satisfies the following conditions:
\begin{equation*}
\begin{array}{c}
  \displaystyle (1)~\supp a_{Q}\subset Q; \quad (2)~\int_{\mathbb{R}^{n}}|a_{Q}(y)|\leq |Q|^{1-\frac{1}{p}};\quad (3)~\int_{\mathbb{R}^{n}}a_{Q}(y)y^{\alpha}dy=0, 0\leq|\alpha|\leq t,
\end{array}
\end{equation*}
where $Q=Q(\bar{y},l)$ is the cube about $\bar{y}$ with sidelength $l>0.$
According to the characterization of the Hardy spaces $H^{p}(\mathbb{R}^{n})$ via the atomic decomposition, it suffices to show that
\begin{equation}
\int_{\mathbb{R}^{n}}|T^{*}a_{Q}(x)|^{p}dx\leq  C,
\end{equation}
for an individual $(p,2,t)$ atom $a_{Q}$,  where constant $C$ independent of $a_{Q}$. We assume without loss of generality the center of the cube $Q$ is at the origin and decompose the operator $T^{*}_{a}$ as (\ref{de}). Then we have
\begin{eqnarray}\label{Sum}
\int_{\mathbb{R}^{n}}|T^{*}a_{Q}(x)|^{p}dx
&\leq&\sum\limits^{\infty}_{j=0}\int_{\mathbb{R}^{n}}|T^{*}_{j}a_{Q}(x)|^{p}dx.
\end{eqnarray}

For the case $l\geq1$, Lemma \ref{L10} (after taking $q=p$) implies that it can be bounded by
\begin{eqnarray*}
\sum\limits^{\infty}_{j=0}\bigg(2^{-jp\big(n(1-\varrho)(\frac{1}{p}-\frac{1}{2})-\frac{n}{2}\max(0,\delta-\varrho)\big)}
+l^{p(\frac{n}{2}-N_{2})}2^{-jp(n(1-\varrho)(\frac{1}{p}-\frac{1}{2})-\frac{n}{2}+\varrho N_{2})}\bigg).
\end{eqnarray*}
Clearly, $n(1-\varrho)(\frac{1}{p}-\frac{1}{2})-\frac{n}{2}\max(0,\delta-\varrho)>0$ and $n(1-\varrho)(\frac{1}{p}-\frac{1}{2})-\frac{n}{2}+\varrho N_{2}>0$(the case $\varrho=0$ is trivial and the case $\varrho\neq0$ can be get by letting $N_{2}$ large enough). So the sum above is convergence.

Next, we put our eyes on the case $l<1$. Break up the sum in (\ref{Sum}) as (\ref{Break}) again.

If $0<\varrho<1$. By Lemma \ref{La2}, Lemma \ref{La1} and Lemma \ref{L10} (after taking $q=p$), we see that it can be bounded by
\begin{eqnarray*}
&&\sum\limits_{2^{j}<l^{-1}}2^{jn(p-1)+jn\frac{t}{2N_{1}}(1-\frac{p}{2})}l^{n(p-1)+n\frac{t}{2N_{1}}(1-\frac{p}{2})}
+\sum\limits_{l^{-1}\leq2^{j}\leq l^{-\frac{1}{\varrho}}}2^{jn(p-1)}l^{n(p-1)}\\
&+&\sum\limits_{l^{-\frac{1}{\varrho}}<2^{j}}\bigg(2^{-jp\big(n(1-\varrho)(\frac{1}{p}-\frac{1}{2})-\frac{n}{2}\max(0,\delta-\varrho)\big)}
+l^{p(\frac{n}{2}-N_{2})}2^{-jp(n(1-\varrho)(\frac{1}{p}-\frac{1}{2})-\frac{n}{2}+\varrho N_{2})}\bigg)
\end{eqnarray*}
It is easy to see that the second term above is convergent since $0<p<1$. Let $t$ large enough, and then we can choose suitable positive integer $2N_{1}>\frac{n(2-p)}{2}$ so that $n(p-1)+n\frac{t}{2N_{1}}(1-\frac{p}{2})>0$ since $1-\frac{p}{2}>0$. So The first term is convergent too. Taking $N_{2}$ large enough, we get the last term is convergent to
$$1+l^{n(1-p)(\frac{1}{\varrho}-1)}\lesssim1.$$

If $\varrho=0$. By Lemma \ref{La2}, Remark \ref{R3} and the same argument as above, we get the desired estimate easily.

Now we consider that $p=1$. The case $0\leq\delta\leq\varrho<1$ has been done by P\"{a}iv\"{a}rinta and Somersalo \cite{Somersalo}. The remaining case $0\leq \varrho<\delta<1$ will be considered only. To this end, break the sum in (\ref{Sum}) as (\ref{Break}) again. The sum for $2^{j}<l^{-1}$ and $2^{j}>l^{-\frac{1}{1-\delta}}$ when $\varrho=0$, and for $2^{j}<l^{-1}$ and $2^{j}>l^{-\frac{1}{\varrho}}$ when $0<\varrho<1$ are convergence by Lemma \ref{La2} and Lemma \ref{L10} (after taking $q=p=1$). By Lemma \ref{La1}, one can not deal with the sum for $l^{-1}\leq2^{j}\leq l^{-\frac{1}{1-\delta}}$ when $\varrho=0$, and for $l^{-1}\leq2^{j}\leq l^{-\frac{1}{\varrho}}$ when $0<\varrho<1$ as above. Because, there is no convergence factor in this lemma when $q=p=1$. One can overcome this problem as the corresponding case in the proof of Theorem \ref{TH2} by following lemmas. So the proof is finished.
\end{proof}

\begin{lem}\label{L6}
Let $Q(x_{0},l)$ be a fixed cube with side length $l<1$. Suppose $0<\varrho<\delta<1$, $a\in S^{-\frac{n}{2}(1-\varrho)}_{\varrho,\delta}$. Then for any $1\leq\lambda\leq\frac{1}{\varrho}$, any positive integer $N>\frac{n}{2}$ and any positive integer $j$ with $l^{-\lambda}\leq2^{j}\leq l^{-\frac{1}{\varrho}}$, we have
\begin{eqnarray*}
\int_{\mathbb{R}^{n}}|T^{*}_{j}a_{Q}(x)|dx
&\lesssim&2^{j\delta}l^{\lambda}+2^{j\frac{n}{2}(\frac{n}{2N}-1)}l^{\frac{n\lambda}{2}(\frac{n}{2N}-1)}.
\end{eqnarray*}
\end{lem}

\begin{lem}
Let $Q(x_{0},l)$ be a fixed cube with side length $l<1$. Suppose $\varrho= 0$, $0<\delta<1$, $a\in S^{-\frac{n}{2}}_{0,\delta}$, then for any $1\leq\lambda\leq\frac{1}{1-\delta}$, any positive integer $N>\frac{n}{2}$ and any positive integer $j$ with $l^{-\lambda}\leq2^{j}\leq l^{-\frac{1}{1-\delta}}$,
\begin{eqnarray*}
\int_{\mathbb{R}^{n}}|T^{*}_{j}a_{Q}(x)|dx
&\lesssim&2^{j\delta}l^{\lambda}+2^{j\frac{n}{2}(\frac{n}{2N}-1)}l^{\frac{n\lambda}{2}(\frac{n}{2N}-1)}.
\end{eqnarray*}
\end{lem}
Theses lemmas can be proved by the main idea in the proof Lemma \ref{L7}. We will only outline the proof of Lemma \ref{L6}.
\begin{proof}[Proof of Lemma \ref{L6}]
Let $Q(x_{i},l^{\lambda})$ be given in the proof of Lemma \ref{L7}.
$$Q(x_{0},l)\subset\cup_{i=1}^{L^{n}}Q(x_{i},l^{\lambda})\subset Q(x_{0},2l).$$
Denote
\begin{eqnarray*}
T^{*}_{j,i}a_{Q}(x)
&=&\int_{\mathbb{R}^n}\int_{\mathbb{R}^n} e^{ i \langle x-y,\xi\rangle}a(x_{i},\xi)\psi(2^{-j}\xi)d\xi a_{Q}(y) dy.
\end{eqnarray*}
We write
\begin{eqnarray*}
&&\int_{\mathbb{R}^{n}}|T^{*}_{j}a_{Q}(x)|dx\leq \sum\limits_{i=1}^{L^{n}}\int_{\mathbb{R}^{n}}|T^{*}_{j}(a_{Q}\chi_{Q(x_{i},l^{\lambda})})(x)|dx\\
&\leq&\sum\limits_{i=1}^{L^{n}}\bigg(\int_{\mathbb{R}^{n}}|T^{*}_{j}(a_{Q}\chi_{Q(x_{i},l^{\lambda})})(x)-T^{*}_{j,i}(a_{Q}\chi_{Q(x_{i},l^{\lambda})})(x)|dx
+\int_{\mathbb{R}^{n}}|T^{*}_{j,i}(a_{Q}\chi_{Q(x_{i},l^{\lambda})})(x)|dx\bigg).
\end{eqnarray*}
Using the similar method as Lemma \ref{L8}($p=2$) and Lemma \ref{L9}($p=2$), one can get
$$\int_{\mathbb{R}^{n}}|T^{*}_{j}(a_{Q}\chi_{Q(x_{i},l^{\lambda})})(x)-T^{*}_{j,i}(a_{Q}\chi_{Q(x_{i},l^{\lambda})})(x)|dx\lesssim2^{j\delta}l^{n(\lambda-1)+\lambda}$$
and
$$\int_{\mathbb{R}^{n}}|T^{*}_{j,i}(a_{Q}\chi_{Q(x_{i},l^{\lambda})})(x)|dx\lesssim 2^{j\frac{n}{2}(\frac{n}{2N}-1)}l^{\frac{n\lambda}{2}(\frac{n}{2N}-1)+n(1-\lambda)},$$
which gives the desired estimate immediately.
\end{proof}

\bibliographystyle{Plain}

\begin{thebibliography}{10}
\bibitem{Hounie} J. \'{A}lvarez, J. Hounie, Estimates for the kernel and continuity properties of pseudo-differential operators. Arkiv f\"{o}r matematik, 1990, 28(1): 1-22.


\bibitem{Alvarez1}  J. \'{A}lvarez, M. Milman, $H^{p}$ continuity properties of Calder\'{o}n-Zygmund-type operators. Journal of mathematical analysis and applications, 1986, 118(1): 63-79.

\bibitem{Alvarez2} J. \'{A}lvarez, M. Milman, Vector valued inequalities for strongly singular Calder\'{o}n-Zygmund operators. Revista matem\'{a}tica iberoamericana, 1986, 2(4): 405-426.

\bibitem{Bergh} J. Bergh and J. L\"{o}fstr\"{o}m, Interpolation Spaces. An Introduction., Springer, Berlin (1976).

\bibitem{Beltran} D. Beltran, L. Cladek, Sparse bounds for pseudodifferential operators. Journal d'Analyse Math\'{e}matique, 2020, 140: 89-116.


\bibitem{Bagchi1} S. Bagchi, R. Basak, R. Garg R, A. Ghosh, Sparse bounds for pseudo-multipliers associated to Grushin operators, I, Journal of Fourier Analysis and Applications, 2023, 29(3): 1-38.



\bibitem{Bagchi2} S. Bagchi, R. Basak, R. Garg R, A. Ghosh, Sparse Bounds for Pseudo-multipliers Associated to Grushin Operators, II, The Journal of Geometric Analysis, 2024, 34(2): 34.

\bibitem{Chanillo} S. Chanillo, A. Torchinsky, Sharp function and weighted $L^{p}$ estimates for a class of pseudo-differential operators, Arkiv f\"{o}r matematik, 1986, 24(1): 1-25.

\bibitem{Coifman} R. Coifman, Y. Meyer,  Au del\`{a} des op\'{e}rateurs pseudo-diff\'{e}rentiels. Ast\'{e}risque, 1978, 57.


\bibitem{Ragusa2} F. Deringoz, V.S. Guliyev, M.N. Omarova, M.A. Ragusa, Calder\'{o}n-Zygmund operators and their commutators on generalized weighted Orlicz-Morrey spaces. Bulletin of Mathematical Sciences, 2023, 13(01): 2250004.


\bibitem{Stein1} C. Fefferman, E. Stein. $H^p$ spaces of several variables. Acta Mathematica, 1972, 129(3-4): 137-193.

\bibitem{Fefferman} C. Fefferman, E. M. Stein, $H^{p}$ spaces of several variables. Acta mathematica, 1972, 129(1): 137-193.

\bibitem{Frazier} M. Frazier, R. Torres, G. Weiss, The Boundedness of Calder\'{o}n-Zygmund Operators on the Spaces $\dot {F}^{\alpha, q} _p$. Revista matem\'{a}tica iberoamericana, 1988, 4(1): 41-72.
    

\bibitem{Ragusa1} V.S. Guliyev, M.N. Omarova, M.A. Ragusa, Characterizations for the genuine Calder\'{o}n-Zygmund operators and commutators on generalized Orlicz-Morrey spaces. Advances in Nonlinear Analysis, 2023, 12(1): 20220307.
    

\bibitem{Goldberg} D. Goldberg, A local version of real Hardy spaces. Princeton University, 1978.


\bibitem{Hart} J. Hart, G. Lu, Hardy space estimates for Littlewood-Paley-Stein square functions and Calder\'{o}n-Zygmund operators. Journal of Fourier Analysis and Applications, 2016, 22(1): 159-186.

\bibitem{Hormander3} L. H\"{o}rmander, On the $L^{2}$ continuity of pseudo-differential operators. Communications on Pure and Applied Mathematics, 1971, 24(4): 529-535.


\bibitem{Hounie2} J. Hounie, On the $L^{2}$ continuity of pseudo-differential operators. Communications in partial differential equations, 1986, 11(7): 765-778.


\bibitem{Hormander2} L. H\"{o}rmander, Pseudo-differential operators and hypoelliptic equations, Singular integrals (Proc. Sympos. Pure Math., Vol. X, Chicago, Ill., 1966), Amer. Math. Soc., Providence, R.I., 1967, 138-183.



\bibitem{Journe} J. Journ\'{e}, Calder\'{o}n-Zygmund operators, pseudo-differential operators and the Cauchy integral of Calder\'{o}n. Springer, 2006.

\bibitem{Kenig} C. Kenig, W. Staubach, $\Psi$-pseudodifferential operators and estimates for maximal oscillatory integrals. Studia mathematica, 2007, 183: 249-258.

\bibitem{MiyachiY} A. Miyachi, K. Yabuta, Sharp function estimates for pseudo-differential operators of class $S^{m}_{\varrho,\delta}$, Bulletin of the Faculty of Science, Ibaraki University. Series A, Mathematics, 1987, 19: 15-30.



\bibitem{Miller} N. Miller, Weighted Sobolev spaces and pseudodifferential operators with smooth symbols. Transactions of the American Mathematical Society, 1982, 269(1): 91-109

\bibitem{Miyachi} A. Miyachi, On some singular Fourier multipliers, Journal of the Faculty of Science, the University of Tokyo. Sect. 1 A, Mathematics, 1981, 28(2): 267-315.



\bibitem{Michalowski} N. Michalowski, D. Rule, W. Staubach, Weighted $L^{p}$ boundedness of pseudodifferential operators and applications. Canadian mathematical bulletin, 2012, 55(3): 555-570.

\bibitem{Michalowski1} N. Michalowski, D. Rule, W. Staubach, Weighted norm inequalities for pseudo-pseudodifferential operators defined by amplitudes. Journal of Functional Analysis, 2010, 258(12): 4183-4209.




\bibitem{Park}  B. Park, N. Tomita, Sharp maximal function estimates for linear and multilinear pseudo-differential operators. Journal of Functional Analysis (2024): 110661.


\bibitem{Park1} J. Park, Boundedness of pseudo-differential operators of type $(0,0)$ on Triebel-Lizorkin and Besov spaces. Bulletin of the London Mathematical Society, 2019, 51(6): 1039-1060.

\bibitem{Park2} B. Park, On the boundedness of pseudo-differential operators on Triebel-Lizorkin and Besov spaces. Journal of Mathematical Analysis and Applications, 2018, 461(1): 544-576.

\bibitem{Park3} B. Park, N. Tomita, Sharp Maximal function estimates for Multilinear pseudo-differential operators of type (0,0).  arXiv preprint arXiv:2405.02093 (2024).

\bibitem{Somersalo} L. P\"{a}iv\"{a}rinta, E. Somersalo, A Generalization of the Calder\'{o}n-Vaillancourt Theorem to $L^{p}$ and $h^{p}$. Mathematische Nachrichten, 1988, 138(1): 145-156.
    
    
\bibitem{Ragusa3} M.A. Ragusa, On Some Local and Nonlocal Variational Problems. Generalized Functions Online Workshop. Cham: Springer Nature Switzerland, 2021: 321-331.
  
    

\bibitem{Stein} E. Stein, Harmonic Analysis: Real-Variable Methods, Orthogonality, and Oscillatory Integrals, volume 43 of Princeton Mathematical Series. Princeton University Press, NJ, 1993.

\bibitem{Taibleson} M. Taibleson, G. Weiss, The molecular characterization of certain Hardy spaces. Soci\'{e}t\'{e} Math\'{e}matique de France, 1980.



\bibitem{Torres} H. Torres, Boundedness results for operators with singular kernels on distribution spaces. Memoirs of the American Mathematical Society, 1991,90(442).


\bibitem{Wang} R. Wang, C. Li, On the $L^{p}$-boundedness of several classes of Pseudo-Differential operators. Chinese Annals of Mathematics, Sseries B, 5 (1984) 193-214.

\bibitem{CW} G. Wang, W. Chen, A pointwise estimate for pseudo-differential operators. Bulletin of Mathematical Sciences, 2023, 13(02): 2250001.


\bibitem{W} G. Wang, Sharp function and weighted $L^{p}$ estimates for pseudo-differential operators with symbols in general H\"{o}rmander classes. arXi preprint arXi:2206.09825, 2022.









\end{thebibliography}

\end{sloppypar}
\end{document}